\documentclass[11pt,reqno]{amsart}
\usepackage{xifthen}
\usepackage{subfig}
\usepackage{pkgfile}

\usepackage{amssymb,amsfonts,amsmath,amsthm,amscd,dsfont,mathrsfs}
\usepackage{graphicx,float,psfrag,epsfig}
\usepackage{wrapfig}
\usepackage{enumitem}

\DeclareMathAlphabet{\mathpzc}{OT1}{pzc}{m}{it}

\newtheorem{propo}{Proposition}[section]
\newtheorem{lemma}[propo]{Lemma}
\newtheorem{definition}[propo]{Definition}
\newtheorem{coro}[propo]{Corollary}
\newtheorem{theo}[propo]{Theorem}

\newtheorem{rmk}[propo]{Remark}

\def\de{{\rm d}}

\newcommand{\sign}{\text{sign}}
\newcommand{\reals}{{\mathds R}}

\newcommand{\eqnsection}{\renewcommand{\theequation}{\thesection.\arabic{equation}}
      \makeatletter \csname @addtoreset\endcsname{equation}{section}\makeatother}
\def\eps{\epsilon}
\def\l|{\left|\left|}
\def\r|{\right|\right|}

\def\P{\mathbb P}
\def\N{\mathbb N}
\def\E{\mathbb E}

\def\ve{\varepsilon}

\def\de{{\rm d}}
\def\reals{{\mathds R}}
\def\us{\underline{\sigma}}
\def\ush{\underline{\widetilde{\sigma}}}
\def\sh{\widetilde{\sigma}}
\def\SK{\mbox{\tiny\rm SK}}
\def\D{\mbox{\tiny\rm D}}

\def\cut{{\rm cut}}

\def\Par{{\sf P}}
\def\Pg #1{G_{n, #1}^{\mbox{\tiny\rm Poiss}}}
\def\Rg #1{G^{\mbox{\tiny\rm Reg}}(n, #1)}

\def\Pc #1{G ^{\tiny\sf{clon}}(n,#1)} 

\def\Erb{E_{\tiny{\sf RB}}}
\def\Gbb{G_{\tiny{\sf BB}}}
\def\Grr{G_{\tiny{\sf RR}}}
\def\Grb{G_{\tiny{\sf RB}}}

\def\tGrr{\widetilde{G}_{\tiny {\sf RR}}} 

\def\hU{\overline{U}}

\def\mcut{{\sf mcut}}
\def\Mcut{{\sf MCUT}}
\def\J{\mathbf{J}} 
\def\z{\mathbf{z}} 
\def\tJ{\widetilde{J}}
\def\tbJ{\widetilde{\J}}
\def\ER{Erd\H{o}s-R\'enyi } 
\def\var{ \rm Var} 
\def\MaxCut{{\sf MaxCut}} 
\def\Deg{{\mathbf Z}}
\def\Degsec{{\mathbf Y}} 

\def\cL{{\mathcal L}}

\newcommand{\abbr}[1]{{\tt{\uppercase{\scalebox{0.85}{#1}}}}}
\title{Extremal Cuts of Sparse Random Graphs}

\author{Amir Dembo$^*$}
\address{Department of Statistics and Mathematics, Stanford University, California, USA, \tt{amir@math.stanford.edu}}
\thanks{${}^*$Research partially supported by NSF grant DMS-1106627.}

\author{Andrea Montanari$^\dagger$}
\address{Department of Electrical Engineering and Statistics, Stanford University, California, USA,  \tt{montanari@stanford.edu}}
\thanks{${}^\dagger$Research partially supported by NSF grants 
DMS-1106627 and CCF-1319979.}

\author{Subhabrata Sen} 
\address{Department of Statistics, Stanford University, California, \tt{ssen90@stanford.edu}} 

\date{\today}



\begin{document}

\begin{abstract}
For \ER  random graphs with average degree $\gamma$, and uniformly random 
$\gamma$-regular graph on $n$ vertices, we prove that 
with high probability the size of both the Max-Cut and   
maximum bisection are
$n(\frac{\gamma}{4} + \Par_* \sqrt{\frac{\gamma}{4}} + o(\sqrt{\gamma})) + o(n)$
while the size of the 
minimum bisection is 
$n(\frac{\gamma}{4}-\Par_*\sqrt{\frac{\gamma}{4}} + o(\sqrt{\gamma})) + o(n)$.
Our derivation relates the free energy of the 
anti-ferromagnetic Ising model on such graphs to that of the 
Sherrington-Kirkpatrick model, with 
$\Par_* \approx 0.7632$ standing 
for the ground state energy of the latter,
expressed analytically via Parisi's formula.
\end{abstract}

\maketitle

\section{Introduction}
\label{introduction}
Given a graph $G=(V,E)$, a bisection of $G$ is a partition of its vertex set
$V=V_1\cup V_2$ such that the two parts have the same cardinality
(if $|V|$ is even) or differ by one vertex (if $|V|$ is odd). The cut 
size of any partition is defined as the number 
of edges $(i,j)\in E$ such that $i\in V_1$, and $j\in V_2$. 
The minimum (maximum) bisection of $G$ is defined as the 
bisection with the smallest (largest)
size and we will denote this size by $\mcut(G)$ (respectively $\Mcut(G)$). The related 
Max-Cut problem seeks to partition the vertices into two parts such that the 
cut size is maximized. We will denote the size of the Max-Cut by
$\MaxCut(G)$. 
The study of these features is fundamental in combinatorics and theoretical computer science. 
These properties are also critical for a number of practical applications. For example,
minimum bisection is relevant for a number of graph layout and embedding problems
\cite{DiazReview}. For practical applications of Max-Cut, see \cite{PT}.  On the other hand,
it is hard to even approximate these quantities in polynomial time (see, for instance \cite{Hastad,Halperin,Feige,Khot}). 

The average case analysis of these features is also of considerable
interest.  
For example, the study of random graph bisections is motivated by the desire
to justify and understand various graph partitioning
heuristics. Problem instances are usually chosen from the \ER and
uniformly random regular graph ensembles. We recall that an \ER random
graph $G(n,m)$ on $n$ vertices with $m$ edges is a
graph formed by choosing $m$ edges uniformly at random
among all the possible edges. A $\gamma$-regular random graph on
$n$-vertices ${\Rg \gamma}$ 
is a graph drawn uniformly from the set of all graphs on $n$-vertices
where every vertex has degree 
$\gamma$ (provided $\gamma n$ is even). 
See \cite{Bollobas,Janson,Hofstad} for detailed analyses of these graph ensembles.

Both min bisection and Max-Cut undergo phase transitions on the \ER graph 
$G(n, [\gamma n])$. For $\gamma <\log 2$, the largest component has less than $n/2$ vertices and minimum bisection is $O(1)$ asymptotically as $n\to \infty$ while above this threshold, the largest component has size greater than $n/2$ and min bisection is $\Omega(n)$ \cite{Luczak}. Similarly, Max-Cut exhibits a phase transition at $\gamma = 1/2$. The difference between the number of edges and Max-Cut size is $\Omega(1)$ for $\gamma <1/2$, while it is $\Omega(n)$ when $\gamma>1/2$ \cite{Coppersmith}. The distribution of the Max-Cut size in the critical scaling window was determined in \cite{daude}. In this paper, we work in the $\gamma \to \infty$ regime, so that both min-bisection and Max-Cut are $\Omega(n)$ asymptotically. 

Diverse techniques have been employed in the analysis of  minimum and maximum bisection for random graph ensembles. For example, 
\cite{Bollobas84} used the Azuma-Hoeffding inequality to establish that 
$\frac{\gamma}{4}-\sqrt{\frac{\gamma\log 2}{4}}\leq\mcut({\Rg
  \gamma})/n\le \frac{\gamma}{4}+\sqrt{\frac{\gamma\log
    2}{4}}$. 

Spectral relaxation based approaches can also be used to bound these
quantities.  
These approaches observe that the  minimum and maximum bisection problem can be written as optimization 
problems over variables $\sigma_i\in\{-1,+1\}$ associated to the
vertices of the graph. By relaxing the integrality
constraint to an $L_2$ constraint the resulting problem can be solved 
through spectral methods.
For instance, the minimum bisection is bounded as follows
(here $\Omega_n\subseteq \{-1,+1\}^n$ is the set of $(\pm 1)$-vectors with
$\sum_{i =1}^n\sigma_i=0$, assuming for simplicity $n$ even)
\begin{equation}\label{eq:rep-mMcut}
\mcut(G) = 
\min_{\us\in \Omega_n} \Big\{ 
\frac{1}{4}\sum_{(i,j)\in E} (\sigma_i-\sigma_j)^2\Big\}
= \frac{1}{2} \min_{\us\in \Omega_n}
            \{ \us \cdot (\cL_G \us) \} \ge \frac{1}{2}\,
                 \lambda_2(\cL_G)\, .
\end{equation}
Here $\cL_G$ is the Laplacian of $G$, with eigenvalues
$0=\lambda_1(\cL_G)\le \lambda_2(\cL_G)\le \cdots \lambda_n(\cL_G)$. 
For  regular graphs, using the result of \cite{Friedman}, this implies
$\mcut({\Rg \gamma})/n\ge \frac{\gamma}{4}-\sqrt{\gamma-1}$. 
However for \ER graphs $\lambda_2(\cL_G)=o(1)$ vanishes with $n$
\cite{khorunzhiy2006lifshitz} and hence this approach fails. 
A similar spectral relaxation yields, for regular graphs,
, $\Mcut({\Rg \gamma})/n\le \frac{\gamma}{4}+
\sqrt{\gamma-1}$, but fails for \ER graphs. Non-trivial spectral
bounds on \ER
graphs can  be derived, for instance, from \cite{feige2005spectral,coja2007laplacian}.

An  alternative approach consists in analyzing algorithms to minimize
(maximize) the cut size. This provides upper bounds on 
$\mcut(G)$ (respectively, lower bounds on $\Mcut(G)$). For 
instance, 
\cite{Alon97} proved that all regular graphs 
have $\mcut(G)/n\le \frac{\gamma}{4}-\sqrt{\frac{9\gamma}{2048}}$ for all $n$ large enough (this method was further developed in \cite{Diaz07}).

Similar results have been established for the max-cut problem on \ER
random graphs. In a recent breakthrough paper,
\cite{BGT} establish that there exists $\mathcal{M}(\gamma)$ such that $\MaxCut(G(n, [\gamma n]))/n \pto \mathcal{M}(\gamma)$
and following upon it
\cite{GL} prove that $\mathcal{M}(\gamma)\in [\gamma/2+0.47523 \sqrt{\gamma}, \gamma/2 + 0.55909\sqrt{\gamma}]$.

To summarize, the general flavor of these results is that if $G$ is an \ER or a random regular graph on $n$ vertices with $[\gamma n/2]$ edges, then 
$\mcut(G)/n = \gamma/4 - \Theta(\sqrt{\gamma})$ 
while $\Mcut(G)/n$ and $\MaxCut(G)/n$ behave asymptotically like
$\gamma/4 + \Theta(\sqrt{\gamma})$.  In other words, the relative spread of cut widths around its average is of order $1/\sqrt{\gamma}$. Despite 30 years of research in combinatorics and random graph theory, even the leading behavior of such a spread is undetermined.

On the other hand, there are detailed and intriguing predictions in
statistical physics--- based mainly on the 
non-rigorous cavity method \cite{mezard2009information}, which relate the behavior of these
features to that of mean field spin glasses.
 From a statistical physics perspective, determining the minimum
 (maximum)  bisection is equivalent to finding the ground state energy
 of the  ferromagnetic (anti-ferromagnetic) Ising model constrained to
 have zero magnetization (see \cite{percus2008peculiar} and the
 references therein). Similarly, the Max-Cut is naturally associated
 with the ground state energy of an anti-ferromagnetic Ising model on
 the graph. 
The cavity method then suggests a surprising conjecture \cite{zdeborova}
that, with high probability,
$\Mcut({\Rg \gamma}) = \MaxCut({\Rg \gamma})+o(n) =n\gamma/2-\mcut({\Rg \gamma}) +o(n)$. 

The present paper  bridges this gap, by partially confirming some of
the  physics predictions and provides estimates of these features which are sharp up to corrections of 
order $n o(\sqrt{\gamma})$. Our estimates are expressed in terms of the 
celebrated Parisi formula for the free-energy of the Sherrington 
Kirkpatrick spin glass, and build on its recent proof by Talagrand.
 In a sense, these results explain the difficulty
encountered by classical combinatorics techniques in attacking 
this problem. In doing so, we develop a new approach based on an interpolation 
technique from the theory of mean field spin glasses 
\cite{GuerraToninelliSK,GuerraToninelliDiluted,TalagrandBook}.
So far this technique has been used in combinatorics
only to prove bounds \cite{FranzLeone}. We combine and extend these ideas, crucially utilizing properties of both the Poisson and Gaussian distributions to derive an 
asymptotically sharp estimate.

\subsection{Our Contribution} 
To state our results precisely, we proceed with a short review 
of the Sherrington-Kirkpatrick (SK) model of spin glasses. 
This canonical example of a mean field spin glass has been studied
extensively by physicists \cite{MPV}, and seen an explosion of 
activity in mathematics
following Talagrand's proof of the Parisi formula, leading to 
better understanding of the SK model and its generalizations (c.f. 
the text \cite{Pan} for an introduction to the subject).

The SK model is a (random) probability distribution on the hyper-cube
$\{-1,+1\}^n$ which assigns mass proportional to $\exp(\beta
H^{\SK}(\us))$ to each `spin configuration' 
$\us \in \{-1,+1\}^n$. The parameter $\beta>0$ is
interpreted as the inverse temperature, with $H^{\SK}(\cdot)$ 
called the Hamiltonian of the model. The collection 
$\{H^{\SK}(\us): \us \in \{-1,+1\}^n\}$ is a 
Gaussian process on 
$\{-1,+1\}^n$ with mean 
$\E[H^{\SK}(\us)] = 0$ and covariance $\E\{H^{\SK}(\us)H^{\SK}(\us')\}=
\frac{1}{2n}\, (\us\cdot\us')^2$. This
process is usually constructed by
\begin{eqnarray}\label{eq:def-SK-hamilton}
H^{\SK}(\us) = -\frac{1}{\sqrt{2n}} \sum_{i,j=1}^n J_{ij}\sigma_i\sigma_j\, ,
\end{eqnarray}
with $\{J_{ij}\}$ being $n^2$ independent standard Gaussian variables,
and we are mostly interested in the ground state energy of the SK model. That is,
the expected (over $\{J_{ij}\}$) minimum (over $\us$), of the Gaussian process 
$H^{\SK}(\us)$ introduced above.  
\begin{definition}
Let $\mathcal{D}_{\beta}$ be the space of non-decreasing, right-continuous non-negative functions 
$x:[0,1]\to [0,\beta]$.
The \emph{Parisi functional} at inverse temperature $\beta$ is the 
function $\Par_{\beta}:\mathcal{D}_{\beta}\to \reals$ defined by
\begin{eqnarray}
\Par_{\beta}[x] = f(0,0;x)-\frac{1}{2}\int_{0}^1 q\, x(q)\, \de q\, ,
\end{eqnarray}
where $f:[0,1]\times \reals\times \mathcal{D}_{\beta}\to\reals$, 
$(q,y,x)\mapsto f(q,y;x)$ is 
the unique weak solution of the 
\abbr{PDE} with boundary condition
\begin{eqnarray}
\frac{\partial f}{\partial q}+\frac{1}{2}\, \frac{\partial^2 f}{\partial y^2}
+\frac{1}{2}\, x(q)\, \left(\frac{\partial f}{\partial y}\right)^2=0\, ,
\qquad 
f(1,y;x)=(1/\beta)\log (2\cosh (\beta y)) \,
\label{eq:ParisiEq}
\end{eqnarray}
among all continuous functions $f(q,y)$ such that  
$\frac{\partial f}{\partial y} \in L^2([0,1] \times \reals)$.

The Parisi replica-symmetry-breaking prediction for the SK model is
\begin{align}
\Par_{*,\beta} \equiv \inf\{\Par_{\beta}[x]:\, x\in
  \mathcal{D}_{\beta}\}\, .
\end{align}
\end{definition}
We refer to \cite[Proposition 7]{jagannath_tobasco15} for the 
uniqueness of such a solution of \eqref{eq:ParisiEq}, and to 
\cite{auffinger2014parisi} for the strict convexity of  
$x \mapsto \Par_{\beta}[x]$, which implies the existence 
of a unique global minimizer of $\Par_{*,\beta}$.
We are interested here in the zero-temperature limit
\begin{align}\label{eq:SK-zero-temp-lim}
\Par_*\equiv \lim_{\beta \to \infty}\Par_{*,\beta}\, ,
\end{align}
which exists because the free energy density (and hence
$\Par_{*,\beta}$, by \cite{TalagrandFormula}), is uniformly 
continuous in $1/\beta$.
It follows from the Parisi Formula \cite{TalagrandFormula}, that 
\begin{align}
\label{SK_groundstate}
\lim_{n \to \infty} n^{-1} \E[\max_{\us} \{ H^{\SK}(\us) \}] = \Par_* \,.
\end{align}
The partial differential equation
(\ref{eq:ParisiEq}) can be solved numerically 
to high precision, resulting with the numerical evaluation of
$\Par_* =
0.76321\pm 0.00003$ \cite{StatMech}, whereas using the replica symmetric bound
of \cite{guerra2003broken}, it is possible to prove that $\Par_*\le
\sqrt{2/\pi}\approx 0.797885$.

We next introduce some additional notation necessary for stating our results. Throughout the paper, $O(\cdot)$, $o(\cdot)$, and $\Theta(\cdot)$ stands for the usual $n \to \infty$ asymptotic, while $O_{\gamma}(\cdot)$, $o_{\gamma}(\cdot)$ 
and $\Theta_{\gamma}(\cdot)$ are used
to describe the $\gamma \to \infty$ asymptotic regime. We say that a sequence of events $A_n$ occurs with high probability (w.h.p.) if $\P(A_n) \to 1$ 
as $n\to \infty$. Finally, for random $\{X_n\}$ and 
non-random $f: \mathbb{R}^+ \to \mathbb{R}^+$, we say 
that $X_n = o_{\gamma}(f(\gamma))$ w.h.p. as $n\to \infty$ if 
there exists non-random $g(\gamma) = o_{\gamma}(f(\gamma))$ such 
that the sequence $A_n = \{ |X_n| \leq g(\gamma)\}$ occurs w.h.p. 
(as $n\to \infty$).  

Our first result provides estimates of the minimum and maximum bisection of \ER random graphs in terms of the SK quantity $\Par_*$ of \eqref{eq:SK-zero-temp-lim}. 
\begin{theo}
\label{ER}
We have, w.h.p. as $n\to \infty$, that 
\begin{align}
\frac{\mcut(G(n, [\gamma n]))}{n} &= \frac{\gamma}{2}- \Par_* \sqrt{\frac{\gamma}{2}} + o_{\gamma}(\sqrt{\gamma})\, ,\label{mcut_er} \\
 \frac{\Mcut(G(n, [\gamma n] ))}{n} &= \frac{\gamma}{2}+ \Par_* \sqrt{\frac{\gamma}{2}} + o_{\gamma}(\sqrt{\gamma}).
\end{align}
\end{theo}
 \begin{rmk}
\label{er_equivalence} Recall the \ER random graph 
$G_I(n ,p_n)$, where each edge is independently
included with probability $p_n$. Since the number of edges in 
$G_I(n,\frac{2\gamma}{n})$ is concentrated around $\gamma n$, 
with fluctuations of $O(n^{1/2+\ve})$ w.h.p. for any $\epsilon >0$, 
for the purpose of Theorem \ref{ER} the random graph $G_I(n,\frac{2\gamma}{n})$
has the same asymptotic behavior as $G(n,[\gamma n])$.
\end{rmk}

\begin{rmk}
The physics interpretation of Theorem \ref{ER} is that a zero-magnetization constraint 
forces a ferromagnet on a random graph to be in a spin glass phase.
This phenomenon is expected to be generic for models on non-amenable graphs 
(whose surface-to-volume ratio is bounded away from zero), 
in staggering contrast with what happens on amenable graphs (e.g.
regular lattices), where such zero magnetization constraint 
leads to a phase separation. 
\end{rmk} 

We next outline the strategy for proving Theorem \ref{ER} 
(with the detailed proof provided in Section \ref{interpolation}).
For graphs $G= (V,E)$, with vertex set $V=[n]$ and $n$ even, 
we write $\us\in\Omega_n$ if the assignment of binary variables
$\us = (\sigma_1,\dots,\sigma_n)$, $\sigma_i\in \{-1,+1\}$ 
to $V$ is such that $\sum_{i\in V}\sigma_i = 0$.
We further define the Ising energy function
$H_G(\us) = -\sum_{(i,j)\in E}\sigma_i
\sigma_j$, and let $U_-(G)\equiv \min\{H_G(\us):\, \us\in\Omega_n\}$,
$U_+(G)\equiv \max\{H_G(\us):\, \us\in\Omega_n\}$. It is then 
clear that
\begin{eqnarray}
\mcut(G) = \frac{1}{2}\, |E| + \frac{1}{2}\,U_-(G)\,,
\;\;\;\;\;\;
\Mcut(G) = \frac{1}{2}\, |E| + \frac{1}{2}\,U_+(G)\,.
\label{eq:BasicIdentity1}
\end{eqnarray}
In statistical mechanics $\us$ is referred to as a `spin configuration'
and $U_{-}(G)$ (respectively $U_+(G)$), its `ferromagnetic
(anti-ferromagnetic) ground state energy.'

The expected cut size of a random partition is taken care of by the term
$\frac{1}{2}|E|$, whereas standard concentration inequalities imply 
that $U_+(G)$ and $U_-(G)$ are tightly concentrated around their expectation 
when $G$ is a sparse \ER random graph. Therefore, it suffices to prove that 
as $n \to \infty$ all limit points of $n^{-1} \E [ U_{\pm}(G)]$ 
are within $o_{\gamma}(\sqrt{\gamma})$ of $\pm \Par_*\sqrt{2\gamma}$.
Doing so is the heart of the whole argument, and it is achieved 
through the interpolation technique of 
\cite{GuerraToninelliSK,GuerraToninelliDiluted}. 
Intuitively, we replace the graph $G$ 
by a complete graph with random edge weights $J_{ij}/\sqrt{n}$
for $J_{ij}$ independent standard normal random variables, and 
prove that the error induced on $U_{\pm}(G)$ by this replacement is 
bounded (in expectation) by $n\,o_{\gamma}(\sqrt{\gamma})$. Finally, we 
show that the maximum and minimum cut-width of such weighted complete graph,
do not change much when optimizing over all partitions $\us \in \{-1,+1\}^n$
instead of only over the balanced partitions $\us\in\Omega_n$.
Now that the equi-partition constraint has been relaxed, the problem 
has become equivalent to determining the ground state energy of the 
SK spin glass model, which is solved by taking the `zero temperature'
limit of the Parisi formula (from \cite{TalagrandFormula}).

The next result extends Theorem \ref{ER} to $\gamma$-regular random graphs. 
\begin{theo}
\label{d-reg}
We have, w.h.p. as $n\to \infty$, that  
\begin{align}
\frac{\mcut({\Rg \gamma})}{n} &= \frac{\gamma}{4}- \Par_* \sqrt{\frac{\gamma}{4}} + o_{\gamma}(\sqrt{\gamma})\, , \\
\frac{\Mcut({\Rg \gamma})}{n} &= \frac{\gamma}{4}+ \Par_* \sqrt{\frac{\gamma}{4}} + o_{\gamma}(\sqrt{\gamma}).
\end{align}
\end{theo}
The average degree in an \ER graph $G(n, [\gamma n])$ is $2\gamma$ 
so Theorems \ref{ER} and \ref{d-reg} take the same form in terms of average degree. 
However, moving from \ER graphs to regular random graphs having the same
number of edges is non-trivial, since the fluctuation of the degree of a
typical vertex in an \ER graph is $\Theta_{\gamma}(\sqrt{\gamma})$. 
Hence, any coupling of these two graph model yields about
$n\sqrt{\gamma}$ different edges, and merely bounding the difference 
in cut-size by the number of different edges, results in 
the too large $\sqrt{\gamma}$ spread. Instead, as detailed in Section \ref{comparison}, our proof of Theorem \ref{d-reg} relies on a 
delicate construction which ``embeds" an \ER graph of average
degree slightly smaller than $\gamma$, into a $\gamma$-regular random 
graph while establishing 
that the fluctuations in the contribution of the additional edges is
only $n o_{\gamma}(\sqrt{\gamma})$. 

Our next result, whose proof is provided in Section \ref{maxcut},
shows that upto the first order,
the asymptotic of the Max-Cut matches that of the Max bisection 
for both \ER and random regular graphs. 
\begin{theo}
\label{maxcut_thm}$~$
\newline
(a) W.h.p. as $n\to \infty$, we have,
\begin{align}
\frac{\MaxCut(G(n, [\gamma n]))}{n} =  \frac{\gamma}{2} + \Par_* \sqrt{\frac{\gamma}{2}} + o_{\gamma}(\sqrt{\gamma}). \nonumber 
\end{align}
(b)
W.h.p. as $n\to \infty$, we have,
\begin{align}
\frac{\MaxCut({\Rg \gamma})}{n} = \frac{\gamma}{4} + \Par_* \sqrt{\frac{\gamma}{4}} + o_{\gamma}(\sqrt{\gamma}). \nonumber 
\end{align}
\end{theo}

\subsection{Application to community detection}

As a simple illustration of the potential applications of
our results, we consider the problem of detecting communities
within the so called `planted partition model', or stochastic block
model. Given parameters $a>b>0$ and even $n$, we denote by 
$G_I(n,a/n,b/n)$ the random graph over vertex set $[n]$, 
such that given a uniformly random balanced partition
$[n] = V_1\cup V_2$, edges $(i,j)$ are independently 
present with probability $a/n$ when either both $i,j\in V_1$ 
or both $i,j\in V_2$, or alternatively present with probability $b/n$ if 
either $i\in V_1$ and $j\in V_2$, or vice versa. Given a random
graph $G$, the community detection problem requires us to 
determine whether the null hypothesis $H_0: G \sim G_I(n,(a+b)/(2n))$ holds,
or the alternative hypothesis $H_1: G\sim G_I(n,a/n,b/n)$ holds.

Under the alternative hypothesis the cut size of the balanced
partition $(V_1,V_2)$ concentrates tightly around $nb/4$. This 
suggests the optimization-based hypothesis testing 
\begin{align}
T_{{\sf cut}}(G;\theta) = \begin{cases}
0 & \mbox{if $\mcut(G)\le \theta$,}\\
1 & \mbox{otherwise}
\end{cases}
\end{align}
and we have the following immediate consequence of Theorem \ref{ER}.
\begin{coro}\label{cor-test}
Let $\theta_n=(b/4)+\eps_n$ with
$\eps_n \sqrt{n}\to\infty$.
Then, the test $T_{{\sf cut}}(\,\cdot\,;\theta_n)$ succeeds
w.h.p. as $n \to \infty$, provided 
$(a-b)^2\ge 8\Par_*^2 (a+b) + o(a+b)$. 
\end{coro}
Let us stress that we did not provide an efficient algorithm for
computing $T_{{\sf cut}}$ (but see \cite{montanari2015semidefinit} for related work that uses
polynomially computable convex relaxations). By contrast, there exist 
polynomially computable tests that succeed w.h.p.
whenever $(a-b)^2>2(a+b)$ and no test can succeed
below this threshold (see
\cite{decelle2011asymptotic,mossel2013proof,massoulie2014community}). 
Nevertheless, the test  $T_{{\sf cut}}$ is so natural that its
analysis is of independent interest, and Corollary \ref{cor-test}
implies that $T_{{\sf cut}}$ is  
sub-optimal by a factor of at most $4\Par_*^2\approx 2.33$.

\section{Interpolation: Proof of Theorem \ref{ER}}
\label{interpolation}
The \ER random graph $G(n,m)$ considers a uniformly chosen 
element from among all \emph{simple} (i.e. having no loops or double edges),
graphs of $n$ vertices and $m$ edges. For $m=[\gamma n]$ and $\gamma$ 
bounded, such simple graph differs in only $O(1)$ edges
from the corresponding multigraph which makes a uniform choice 
while allowing for loops and multiple edges. 
Hence the two models are equivalent for our purpose,
and letting $G(n,[\gamma n])$ denote hereafter the latter multigraph, 
we note that it can be constructed also by sequentially introducing 
the $[\gamma n]$ edges and independently sampling their end-points 
from the uniform distribution on $\{1,\cdots, n\}$. 
We further let $\Pg \gamma$ denote the 
Poissonized random multigraph $G(n,N_n)$ having
the random number of edges $N_n \sim \dPois(\gamma n)$, 
independently of the choice of edges.
Alternatively, one constructs $\Pg \gamma$ by generating 
for $1\leq i,j \leq n$ the i.i.d $z_{ij} \sim \dPois(\frac{\gamma}{n})$ 
and forms the multi-graph on $n$ vertices by taking 
$(z_{ij}+z_{ji})$ as the multiplicity of each edge $(i,j), i \neq j$
(ending with multiplicity $z_{(i,j)} \sim \dPois( \frac{2\gamma}{n})$
for edge $(i,j)$, $i \ne j$ and the multiplicity 
$z_{(i,i)} \sim \dPois ( \frac{\gamma}{n})$
for each loop $(i,i)$, where  
$\{z_{(i,j)}, i < j, z_{(i,i)} \}$ are mutually independent). 
By the tight concentration of the $\dPois(\gamma n)$ law, 
it suffices to prove Theorem \ref{ER} for $\Pg \gamma$, and 
in this section we always take for $G_n$ a random multi-graph 
distributed as $\Pg \gamma$. 

%
\subsection{Spin models and free energy} 
\label{spin_models}
A spin model is defined by the (possibly random) Hamiltonian
$H:\{-1,+1\}^n\to\reals$ and in this paper we often consider spin models 
constrained to have zero empirical magnetization, namely from the set
$\Omega_n=\{ \us \in \{-1,+1\}^n: \sum_{i=1}^n \sigma_i=0\}$. The constrained partition function is then 
$Z_n(\beta) = \sum_{\us \in \Omega_n} e^{-\beta H(\us)}$ with the 
corresponding constrained free energy density 
\begin{eqnarray}\label{eq:def-free-eng}
\phi_n(\beta) \equiv \frac{1}{n} \; \E [ \log Z_n (\beta) ] =
\frac{1}{n}\;\E \Big[\;\log\big\{
\sum_{\us\in \Omega_n}e^{-\beta H(\us)}\big\} \Big]\,.
\end{eqnarray}
The expectation in \eqref{eq:def-free-eng} is over the distribution of the 
function $H(\,\cdot\,)$ (i.e. over
the collection of random variables $\{H(\us)\}$). Depending on the model under consideration, the Hamiltonian (or the free energy) might depend on additional parameters
which we will indicate, with a slight abuse of notation, as 
additional arguments of $\phi_n(\,\cdot\,)$. 

For such spin models we also consider the 
expected ground state energy density
\begin{eqnarray}\label{eq:def-e}
e_n = \frac{1}{n}\,\E [\,\min_{\us\in \Omega_n} \, H(\us)\,]\,,
\end{eqnarray}
which determines the large-$\beta$ behavior
of the free energy density. That is, $\phi_n(\beta) = -\beta\, e_n +o(\beta)$.
We analogously define the maximum energy 
\begin{eqnarray}\label{eq:def-ehat}
\widehat{e}_n = \frac{1}{n}\,\E [\,\max_{\us\in \Omega_n} \, H(\us)\,]\,,
\end{eqnarray}
which governs the behavior of the free energy density as 
$\beta\to -\infty$. That is, $\phi_n(\beta) = -\beta\, \widehat{e}_n +o(\beta)$
(in statistical mechanics it is more customary to change
the sign of the Hamiltonian in such a way that $\beta$ is kept positive).
The corresponding Boltzmann measure on $\Omega_n$ is 
\begin{eqnarray}\label{eq:Bolt-meas}
\mu_{\beta,n}(\us) =\frac{1}{Z_n(\beta)} \exp\{-\beta H(\us)\}\,.
\end{eqnarray}
A very important example of a spin model, 
that is crucial for our analysis is the SK model
having the Hamiltonian $H^{\SK}(\cdot)$ of \eqref{eq:def-SK-hamilton} 
on $\{-1,+1\}^n$ and we also consider that model
constrained to $\Omega_n$ (i.e. subject to zero magnetization constraint).

The second model we consider is the `dilute' ferromagnetic Ising model 
on $\Pg \gamma = (V,E)$, corresponding to the Hamiltonian
\begin{eqnarray}
H^{\D}_{\gamma}(\us) = -\sum_{(i,j)\in E} \sigma_i\sigma_j\,,
\label{eq:DilutedHamiltonian}
\end{eqnarray}
again restricted to $\us \in \Omega_n$.
We use superscripts to indicate the model to which various quantities 
refer. For instance $\phi^{\SK}_n(\beta)$ denotes 
the constrained free energy of the SK model,  
$\phi^{\D}_n(\beta;\gamma)$ is the constrained 
free energy of the Ising model on $\Pg \gamma$,  
with analogous notations used for
the ground state energies $e^{\SK}_n$ and $e^{\D}_n(\gamma)$.

\medskip
The first step in proving Theorem \ref{ER} is to show 
that $\mcut(G_n)$ and $\Mcut(G_n)$ are 
concentrated around their expectations. 
\begin{lemma}
\label{concentration}
Fixing $\ve>0$, we have that
\begin{align*}
\P\left[ \Big| \mcut(G_n) - \E[\mcut(G_n)]\Big| >
  n\ve\right]&=O(1/n) \,\\
   \P\left[ \Big| \Mcut(G_n) - \E[\Mcut(G_n)]\Big| >
  n\ve\right]&=O(1/n) \,.
\end{align*} 
\end{lemma}
\begin{proof} 
Recall \eqref{eq:rep-mMcut} that 
$\mcut(G_n) = \frac{1}{2} |E_n| + \frac{1}{2} U_{-}(G_n)$,
with $|E_n|=N_n \sim \dPois([\gamma n])$. Therefore,
\begin{align*}
\P\left[|\mcut(G_n) - \E[\mcut(G_n)]|>n\ve \right] &\leq \P[|U_{-}(G_n) - \E[U_{-}(G_n)] | > n \ve] + \P[|N_n - \E N_n| > n\ve]\nonumber \\
&\leq\frac{{\var}(U_{-}(G_n))}{n^2\ve^2} +  \frac{{\var}(N_n)}{n^2 \ve^2} 
=  \frac{{\var}(U_{-}(G_n))}{n^2\ve^2}  + O(1/n) \nonumber .
\end{align*}
We complete the proof for $\mcut(G_n)$ by showing 
that ${\var}(U_{-}(G_n)) \leq n\gamma $. 
Indeed, writing $U_{-}(G_n) = f(\mathbf{z})$ for 
$\mathbf{z}= \{z_{ij}, 1\leq i,j \leq n\}$ and i.i.d. 
$z_{ij} \sim \dPois(\gamma/n)$, we let $\mathbf{z}^{(i,j)}$ denote the 
vector formed when replacing $z_{ij}$ in $\mathbf{z}$ by an i.i.d copy $z_{ij}'$. 
Clearly $|f(\mathbf{z}) - f(\mathbf{z}^{(i,j)})| \leq |z_{ij}- z_{ij}'|$. Hence,
by the Efron-Stein inequality \cite[Theorem 3.1]{BLM},
\begin{align*}
{\var}(U_{-}(G_n)) \leq \frac{1}{2} \sum_{i,j} \E[ (f(\mathbf{z})- f(\mathbf{z}^{(i,j)}))^2] \leq \frac{1}{2} \sum_{i,j} \E[ (z_{ij}- z_{ij}')^2] \,,
\end{align*}
yielding the required bound (and the proof for 
$\Mcut(G_n)= \frac{1}{2} N_n + \frac{1}{2} U_{+}(G_n)$ proceeds along 
the same line of reasoning).
\end{proof}

Next, recall that 
$|E_n| \sim \dPois(\gamma n)$ has expectation 
$\gamma n$, while $e_n^{D}=n^{-1} \E [U_-(G_n)]$ and 
$\widehat{e}_n^D = n^{-1} \E [U_+(G_n)]$ 
(see \eqref{eq:def-e} and \eqref{eq:def-ehat}, respectively).
Hence, from the representation \eqref{eq:rep-mMcut} of 
$\mcut(G_n)$ and $\Mcut(G_n)$, we further conclude that
\begin{align}
\frac{1}{n} \E[ \mcut(G_n)] = \frac{\gamma}{2} + \frac{1}{2} e_n^{\D}(\gamma)\, , \;\;\;\;\;\;  
\frac{1}{n} \E[\Mcut(G_n)] = \frac{\gamma}{2} + \frac{1}{2} \widehat{e}_n^{\D}(\gamma)\,. 
\label{expectation_rep}
\end{align}
Combining \eqref{expectation_rep} with Lemma \ref{concentration}, 
we establish Theorem \ref{ER}, once we show that as $n \to \infty$,
\begin{align}\label{eq:enD-limit}
e_n^{\D}(\gamma) &= - \sqrt{2 \gamma} \Par_* + o_{\gamma}(\sqrt{\gamma}) + o(1), \\
\widehat{e}_n^{\D}(\gamma) &=  + \sqrt{2 \gamma} \Par_* + o_{\gamma}(\sqrt{\gamma}) + o(1).
\label{eq:enhatD-limit}
\end{align}
Establishing \eqref{eq:enD-limit} and \eqref{eq:enhatD-limit} is the 
main step in proving Theorem \ref{ER}, and the key to it is the 
following proposition of independent interest.
\begin{propo}\label{lemma:LargeDegree}
There exist constants $A_1, A_2<\infty$ independent of 
$n$, $\beta$ and $\gamma$ such that 
\begin{eqnarray}
\left|\phi_n^{\D}\left(\frac{\beta}{\sqrt{2\gamma}},\gamma\right)-\phi_n^{\SK}(\beta)
\right|\le 
A_1\,\frac{|\beta|^3}{\sqrt{\gamma}}+
A_2\, \frac{\beta^4}{\gamma}\, .\label{eq:BasicInterpolation}
\end{eqnarray}
\end{propo}
We defer the proof of Proposition \ref{lemma:LargeDegree}
to Subsection \ref{section_bound}, where we also apply it to 
deduce the next lemma, comparing the ground 
state energy of a dilute Ising ferromagnet to that of the SK model,
after both spin models have been constrained to have zero magnetization. 
%
\begin{lemma}
\label{bound}
There exist $A=A(\gamma_0)$ finite, such that for all 
$\gamma \ge \gamma_0$ and any $n$,
\begin{equation}\label{eq:en-ehn-bds}
\Big| \frac{e_n^{\D} (\gamma)}{\sqrt{2\gamma}} - e_n^{\SK} \Big| \leq 
A \gamma^{-1/6} \,, 
\qquad 
\Big| \frac{\widehat{e}_n^{\D} (\gamma)}{\sqrt{2\gamma}} + e_n^{\SK} \Big| \leq 
A \gamma^{-1/6} \,.
\end{equation}
\end{lemma}
In view of Lemma \ref{bound}, we get both \eqref{eq:enD-limit} and 
\eqref{eq:enhatD-limit} once we control the difference between
the ground state energies of the unconstrained and constrained to have zero magnetization SK models. This is essentially established by our following
lemma (whose proof is provided in Subsection \ref{section_approximation}). 
\begin{lemma}
\label{approximation}
For any $\delta>0$, w.h.p. 
$0 \le U^{\SK}_n-\hU_n^{\SK} \le n^{\frac{1}{2}+\delta}$, where
\begin{equation}\label{eq:Undef}
\hU_n^{\SK} = \min_{\us \in \{-1,+1\}^n} \{ H^{\SK}(\us) \}
\,,\qquad 
U_n^{\SK}= \min_{\us \in \Omega_n} \{ H^{\SK}(\us) \}\,.
\end{equation} 
\end{lemma}
Indeed, applying Borel's concentration inequality for 
the maxima of Gaussian processes (see \cite[Theorem 5.8]{BLM}), we have that 
for some $c>0$, all $n$ and $\delta>0$,
\begin{align}
\P\left[ \left| \hU_n^{\SK} - \E[ \hU_n^{\SK} ]  \right| > n \delta \right]
&\leq 2 e^{-cn \delta^2} \label{conc1} \, ,\\
\P\left[\left| U_n^{\SK} - \E[ U_n^{\SK}]\right| > n \delta \right] & 
\leq 2 e^{-cn\delta^2}\,. \label{conc2}
\end{align}
Recall that $e_n^{\SK} = n^{-1} \E[ U_n^{\SK}]$, whereas 
$n^{-1} \E[ \hU_n^{\SK} ] \to - \Par_*$ by \eqref{SK_groundstate}.
Consequently, the bounds of \eqref{conc1}, \eqref{conc2} coupled 
with Lemma \ref{approximation} imply that $e_n^{\SK} \to -\Par_*$ 
as $n\to \infty$. This, combined with Lemma \ref{bound} and 
\eqref{expectation_rep}, completes the proof of Theorem \ref{ER}.

%
%
%
\subsection{The interpolation argument} 
\label{section_bound}
We first deduce Lemma \ref{bound} out of Proposition \ref{lemma:LargeDegree}.
To this end, we use the inequalities of Lemma \ref{largebeta}
relating the free energy of a spin model to its ground state energy
(these are special cases of general bounds for models with at most 
$c^n$ configurations, but for the sake of completeness we include their proof).   
\begin{lemma}\label{largebeta}
The following inequalities hold for any $n$, $\beta,\gamma > 0$:
\begin{eqnarray}
\label{eq:beta-inf}
\Big|
e_n^{\D}(\gamma) +
\frac{1}{\beta}\,\phi^{\D}_n(\beta,\gamma) 
\Big|\le  
\frac{\log 2}{\beta}\, ,\;\;\;\;\;\;\;\;
\Big|
e_n^{\SK} + 
\frac{1}{\beta}\,\phi^{\SK}_n(\beta) 
\Big|\le  
\frac{\log 2}{\beta}\, .
\end{eqnarray}
Further, for any $n$, $\beta<0$, $\gamma>0$,
\begin{eqnarray}
\Big|
\widehat{e}_n^{\D}(\gamma)
+
\frac{1}{\beta}\, \phi^{\D}_n(\beta,\gamma) 
\Big|\le  
\frac{\log 2}{|\beta|}\, ,\;\;\;\;\;\;\;\;
\Big|
e_n^{\SK} - 
\frac{1}{\beta}\,\phi^{\SK}_n(\beta) \Big|\le  
\frac{\log 2}{|\beta|}\, .
\label{eq:beta-minus-inf}
\end{eqnarray}
\end{lemma}
\begin{proof}
Let $H_n(\us)$ be a generic Hamiltonian for $\us \in \Omega_n$. 
One then easily verifies that 
\begin{align*}
\frac{\partial}{\partial \beta} \left(\frac{\phi_n(\beta)}{\beta}\right) &= -\frac{1}{n\beta^2} \E[S(\mu_{\beta,n})]  \in \big[ -\frac{\log 2}{\beta^2},0 \big] \,, 
\end{align*}
for the Boltzman measure \eqref{eq:Bolt-meas} and the non-negative
entropy functional $S(\mu)= - \sum_{\us \in \Omega_n} \mu(\us) \log \mu(\us)$ 
which is at most $\log |\Omega_n|$. 
Further, comparing \eqref{eq:def-free-eng} and \eqref{eq:def-e} we see that 
$\beta^{-1} \phi_n(\beta) \to - e_n$ when $\beta \to \infty$ (while 
$n$ is fixed). Consequently, for any $\be>0$, 
\begin{align*}
\Big| e_n + \frac{\phi_n(\beta)}{\beta} \Big| &=
                                                   \Big|\int_{\beta}^{\infty}
                                                   \frac{\partial}{\partial
                                                   u}
                                                   \left(\frac{\phi_n(u)}{u}\right)
                                                   \de u \Big| \leq \frac{\log 2}{\beta} \,.
\end{align*}
We apply this inequality separately to the SK model 
and the diluted Ising model to get the bounds of \eqref{eq:beta-inf}.
We similarly deduce the bounds of \eqref{eq:beta-minus-inf} upon
observing that 
$\beta^{-1} \phi_n(\beta) \to -\widehat{e}_n$ when $\beta \to -\infty$
and recalling that with $\{H^{\SK}_n(\us)\}$ a zero mean 
Gaussian process, necessarily $\widehat{e}_n^{\SK}= - e_n^{\SK}$. 
\end{proof}

\begin{proof}[Proof of Lemma \ref{bound}]
Clearly, for any $n$, $\beta>0$ and $\gamma>0$, 
\begin{align*}
\Big|\frac{e_n^{\D}(\gamma)}{\sqrt{2\gamma}} - e_n^{\SK} \Big|&
\leq  
\left| 
\frac{1}{\sqrt{2\gamma}} e_n^{\D}(\gamma) 
+
\frac{1}{\beta} \phi_n^{\D}(\frac{\beta}{\sqrt{2\gamma}},\gamma) 
\right|
+ \left| \frac{1}{\beta} \phi_n^{\SK}(\beta) - \frac{1}{\beta} \phi_n^{\D}(\frac{\beta}{\sqrt{2\gamma}}, \gamma)\right| 
+ 
\left| e_n^{\SK} + \frac{1}{\beta} \phi_n^{\SK}(\beta) \right| \,. 
\end{align*}
In view of \eqref{eq:beta-inf}, the first and last terms on the \abbr{rhs}
are bounded by $(\log 2)/\beta$. Setting $\beta= \gamma^{1/6}$, 
we deduce from Proposition \ref{lemma:LargeDegree} that 
the middle term on the \abbr{rhs} is bounded by   
$A_1 \gamma^{-1/6} + A_2 \gamma^{-1/2}$, yielding 
the first (left) bound in \eqref{eq:en-ehn-bds}  
(for $A=\log 2 + A_1 + A_2 \gamma_0^{-1/3}$). In case $\beta<0$, 
starting from 
\begin{align*}
\Big|\frac{\widehat{e}_n^{\D}(\gamma)}{\sqrt{2\gamma}} + e_n^{\SK} \Big|&
\leq  
\left| 
\frac{1}{\sqrt{2\gamma}} \widehat{e}_n^{\D}(\gamma) 
+
\frac{1}{\beta} \phi_n^{\D}(\frac{\beta}{\sqrt{2\gamma}},\gamma) 
\right|
+ \left| \frac{1}{\beta} \phi_n^{\SK}(\beta) - \frac{1}{\beta} \phi_n^{\D}(\frac{\beta}{\sqrt{2\gamma}}, \gamma)\right| 
+ 
\left| e_n^{\SK} - \frac{1}{\beta} \phi_n^{\SK}(\beta) \right| \,, 
\end{align*}
and using \eqref{eq:beta-minus-inf},
yields the other (right) bound in \eqref{eq:en-ehn-bds}.
\end{proof}
\begin{proof}[Proof of Proposition \ref{lemma:LargeDegree}]
For $t \in [0,1]$ we consider the interpolating Hamiltonian on $\Omega_n$
\begin{align}\label{eq:int-Ham}
H_n(\gamma, t,\us) := \frac{1}{\sqrt{2\gamma}} 
H_{\gamma(1-t)}^{\D}(\us)  + \sqrt{t} H^{\SK}(\us) \,,
\end{align}
denoting by $Z_n(\beta,\gamma,t)$, $\phi_n(\beta,\gamma,t)$ 
and $\mu_{\beta,n}(\cdot;\gamma,t)$, the 
partition function, free energy density,
and Boltzmann measure, respectively, 
for this interpolating Hamiltonian. Clearly,  
$\phi_n(\beta,\gamma, 0) = \phi_n^{\D} (\frac{\beta}{\sqrt{2\gamma}}, \gamma)$ 
and $ \phi_n(\beta,\gamma,1) = \phi_n^{\SK}(\beta)$. Hence,
\begin{align*}
\left|\phi_n^{\D}(\frac{\beta}{\sqrt{2\gamma}},\gamma) -
  \phi_n^{\SK}(\beta) \right| \leq \int_{0}^{1} \left|\frac{\partial
  \phi_n}{\partial t} (\beta,\gamma,t) \right|\de t 
\end{align*}
and it suffices to show that 
$|\frac{\partial \phi_n}{\partial t}|$ is 
bounded, uniformly over $t \in [0,1]$ and $n$, by 
the \abbr{rhs} of  \eqref{eq:BasicInterpolation}. 
To this end, associate with i.i.d. configurations 
$\{\us^j, j \ge 1\}$ from $\mu_{\beta,n}(\cdot;\gamma,t)$
and $\ell \ge 1$, the multi-replica overlaps
\begin{eqnarray*}
Q_\ell\equiv \frac{1}{n}\sum_{i=1}^n \Big( \prod_{j=1}^{\ell} \sigma_i^j \Big) \,.
\end{eqnarray*}
Then, denoting by $\langle\, \cdot \, \rangle_t$ the expectation 
over such i.i.d. configurations $\{\us^j, j \ge 1\}$, 
it is a simple exercise in spin glass theory (see for example 
\cite{FranzLeone}), to explicitly express the relevant derivatives as
\begin{align}
\frac{\partial \phi_n}{\partial t}(\beta,\gamma, t) &= \left(\frac{\partial \phi_n}{\partial t} \right)_{\SK} + \left(\frac{\partial \phi_n}{\partial t}\right)_{\D} 
\,, \nonumber \\
\left(\frac{\partial \phi_n}{\partial t} \right)_{\SK} &= \frac{\beta^2}{4} (1- \E[\langle Q_2^2 \rangle_t] ) \,,
\label{eq:der-SK}
\\
\left(\frac{\partial \phi_n}{\partial t}\right)_{\D} &= -\gamma \log\cosh \left(\frac{\beta}{\sqrt{2\gamma}}\right)+ \gamma \sum_{\ell=1}^{\infty} \frac{(-1)^\ell}{\ell} \left(  \tanh \frac{\beta}{\sqrt{2\gamma}} \right)^\ell 
\E[\langle Q_\ell^2 \rangle_t ] \,.
\label{eq:der-D} 
\end{align}
For the reader's convenience, we detail the derivation of 
\eqref{eq:der-SK} and \eqref{eq:der-D} in 
Subsection \ref{interpolation_derivation}, and note in passing 
that the expressions on their \abbr{RHS} resemble the derivatives 
of the interpolating free energies obtained in 
the Gaussian and dilute spin glass models, respectively (see \cite{GuerraToninelliSK}, \cite{GuerraToninelliDiluted}). 

Now observe that $|Q_\ell| \le 1$ for all $\ell \ge 2$ and 
$Q_1=0$ on $\Omega_n$, hence
\begin{align*}
\left|\frac{\partial \phi_n}{\partial t} (\beta,\gamma,t) \right| &\leq \gamma \left|\log\cosh \left( \frac{\beta}{\sqrt{2\gamma}}\right) - \frac{\beta^2}{4\gamma} \right| + \frac{\gamma}{2} \left| \left(\tanh \frac{\beta}{\sqrt{2\gamma}} \right) ^2 - \frac{\beta^2}{2\gamma}  \right| + \gamma \sum_{\ell=3}^{\infty} \frac{1}{\ell} \left| \tanh \frac{\beta}{\sqrt{2\gamma}}  \right|^\ell.
\end{align*}
The required uniform bound on $|\frac{\partial \phi_n}{\partial t}|$ 
is thus a direct consequence of the elementary inequalities
$$
|\log\cosh x-\frac{1}{2}x^2|\le C_1 x^4 \,,\quad
|y^2-x^2|\le C_2x^4 \,,\quad
|-\log(1-y)- y-\frac{1}{2} y^2|\le C_3 |x|^3 \,,
$$ 
which hold for some finite $C_1$, $C_2$, $C_3$ and any $y=|\tanh x|$.
\end{proof}

\subsection{Proof of Lemma \ref{approximation}}\label{section_approximation}
Recall that $H^{\SK}(\us) = - \frac{1}{2\sqrt{n}} \us^T \tbJ\us$ where $\tbJ=\{\tilde{J}_{ij}= (J_{ij}+ J_{ji})/\sqrt{2}: 1\leq i,j \leq n\}$ is a 
\abbr{GOE} matrix. Since $\{\tJ_{ii}\}$ do not affect 
$U^{\SK}_n - \hU^{\SK}_n$, we further set all diagonal entries 
of $\tbJ$ to zero.    
By symmetry of the Hamiltonian $H^{\SK}(\,\cdot\,)$, the configuration
$\us^\star$ that achieves the 
unconstrained ground state energy $H^{\SK}(\us^\star)=\hU^{\SK}_n$ is
uniformly random in $\{-1,+1\}^n$. Therefore, 
$S^\star_n := \frac{1}{2} \sum_{i=1}^n\sigma_i^\star$ 
is a centered $\dBin(n,1/2)$ random variable, and by the 
\abbr{LIL} the events $B_n=\{|S^\star_n|\le b_n \}$ 
hold w.h.p. for $b_n:=\sqrt{n\log n}$. By definition 
$\hU^{\SK}_n \ge -\frac{n}{2} \lambda_{\max}(\tbJ/\sqrt{n})$, 
hence the events $C_n = \{\hU^{\SK}_n\ge -2n \}$ also hold w.h.p.
by the a.s. convergence of the largest eigenvalue 
$\lambda_{\max}(\,\cdot\,)$ for Wigner matrices (see 
\cite[Theorem 2.1.22]{AGZ}). 
Consequently, hereafter our analysis is carried out on the 
event $\{B_n \cap C_n\}$ and without loss of generality we can 
and shall further assume that $S^\star_n>0$ is integer (since $n$ is even).

Since $\us^\star$ is a global minimizer of the quadratic form $H^{\SK}(\us)$
over the hyper-cube $\{-1,1\}^n$, necessarily $\sigma_i^\star = \sign(f^\star_i)$ 
for 
$$
f^\star_i := \frac{1}{2 \sqrt{n}} 
\sum_{j=1}^n \tJ_{ij} \sigma_j^\star \,.
$$ 
Consequently, under the event $C_n$,
$$
-2n \le \hU^{\SK}_n = H^{\SK}(\us^\star) = - \sum_{i=1}^n \sigma_i^\star f_i^\star  
= -\sum_{i=1}^n |f_i^\star| \,, 
$$
hence $R^\star := \{i \in [n] : |f_i^\star| \le 6\}$ is of size 
at least $(2/3) n$. Thus, for $n \ge 6 b_n$, under the event 
$B_n \cap C_n$ we can find a collection  
$W^\star \subseteq \{i \in R^\star:\, \sigma_i^\star=+1\}$ 
of size $S^\star_n$ 
and let $\ush \in \Omega_n$ be the configuration obtained   
by setting  $\sh_i = -\sigma_i^\star = -1$ whenever $i\in W^\star$ while 
otherwise $\sh_i=\sigma_i^\star$.
We obviously have then that
\begin{eqnarray}\label{eq:basic-bd-Un}
\hU^{\SK}_n = H^{\SK}(\us^\star)\le U_n^{\SK} \le H^{\SK}(\ush) \,.
\end{eqnarray}
Further, by our choices of $\ush$ and $W^\star \subseteq R^\star$, also
\begin{align}
H^{\SK}(\ush)-H^{\SK}(\us^\star) &= 
\frac{2}{\sqrt{n}}\sum_{i\in W^\star} \sum_{j\in [n]\setminus W^\star}\tJ_{ij} \sigma_j^\star
\nonumber 
\\& \le
4 \sum_{i\in W^\star} |f_i^\star|
+ \frac{4}{\sqrt{n}} \Delta (W^\star)
\le 24 S_n^\star + \frac{4}{\sqrt{n}} \Delta (W^\star)
\,, \label{eq:TwoPartsBound}
\end{align}
where we define, for $W \subseteq [n]$ the corresponding partial sum 
\begin{align*}
\Delta (W) := \sum_{ i, j \in W, i < j} |\tJ_{ij}| \,,
\end{align*}
of $\binom{|W|}{2}$ i.i.d. variables $\tJ_{ij}$. Under the event $B_n$ we have 
that $S_n^\star \le b_n \le y_n :=  \frac{1}{32} n^{1/2 + \delta}$, so 
by \eqref{eq:basic-bd-Un} and \eqref{eq:TwoPartsBound} it suffices 
to show that w.h.p. $\{\Delta(W^\star) \le x_n \}$ for $x_n = \sqrt{n} y_n$.
To this end, note that by Markov's inequality,
for some $c>0$, all $n$ and any fixed $W$ of size $|W| \le b_n$, 
$$
\P( \Delta (W) \ge x_n) \le e^{-x_n} \E[e^{|\tJ|}]^{b_n^2} \le e^{-c x_n} \,.
$$
With at most $2^n$ such $W \subseteq [n]$, we conclude that
$$
\P( \sup\{ \Delta(W) : W \subset [n], |W|\leq b_n \} \le x_n ) \to 1 \,,
$$
and in particular w.h.p. $\{\Delta(W^\star) \le x_n\}$ (under 
$B_n =\{ S_n^\star \le b_n \}$). 

\subsection{The interpolation derivatives} 
\label{interpolation_derivation}
Recall the 
Hamiltonian $H_n(\gamma,t,\us)$
of \eqref{eq:int-Ham}, 
the corresponding 
partition function $Z_n(\beta,\gamma,t)$
and free energy density $\phi_n(\beta,\gamma,t)$.
We view $n^{-1} \log Z_n(\beta, \gamma, t) :=\psi_n (t,\mathbf{z},\mathbf{J})$, 
as a (complicated) function of the Gaussian 
couplings $\J= \{ J_{ij} : 1\leq i,j \leq n\}$ and the Poisson 
multiplicities $\z= \{z_{ij}: 1\leq i,j \leq n\}$. Denoting by
$p(t,\cdot)$ the $\dPois(\gamma(1-t)/n)$ 
probability mass function (\abbr{pmf}) of $z_{ij}$
yields the joint \abbr{pmf} 
$\mathbf{p}(t,\mathbf{z})= \prod_{1\leq i,j \leq n} p(t, z_{ij})$,
and the expression
\begin{align}
\phi_n(\beta,\gamma,t) &=  \E [\psi_n(t,\mathbf{z},\mathbf{J})] =  
\int \psi_n(t,\mathbf{z},\mathbf{J}) \mathbf{p}(t,\mathbf{z}) \de\mu(\mathbf{z},\mathbf{J}) 
\end{align}
where $\mu= (\nu_{\N})^{n^2} \otimes (\nu_{\mathbb{R}})^{n^2}$
for the counting measure $\nu_{\N}$ on $\N$ and 
the standard Gaussian measure $\nu_{\mathbb{R}}$ on $\mathbb{R}$. Thus,
\begin{align}
\frac{\partial \phi_n}{\partial t}(\beta,\gamma, t) &= 
\int \frac{\partial \psi_n}{\partial t}(t,\mathbf{z},\mathbf{J}) \mathbf{p}(t,\mathbf{z}) \de\mu(\mathbf{z},\mathbf{J}) + \int \psi_n(t,\mathbf{z},\mathbf{J}) 
\frac{\partial \mathbf{p}}{\partial t} (t,\mathbf{z}) \de\mu(\mathbf{z},\mathbf{J})
\nonumber \\
&:= 
\left(\frac{\partial \phi_n}{\partial t} \right)_{\SK}
+ \left(\frac{\partial \phi_n}{\partial t} \right)_{D} \,.
\end{align}  
Proceeding to verify the expression \eqref{eq:der-SK}, here
$\frac{\partial H_n}{\partial t} = \frac{1}{2\sqrt{t}} H^{\SK}$
(since $H^{\D}_{\gamma(1-t)}(\cdot)$ depends on $t$ only through the 
\abbr{pmf} of $\mathbf{z}$). Hence,
$$
\frac{\partial}{\partial t} \big[ \log Z_n(\beta,\gamma,t) \big]
= -\beta \Big\langle 
\frac{\partial H_n}{\partial t}(\gamma,t,\us) \Big\rangle_t
= - \frac{\beta}{2 \sqrt{t}} \Big\langle H^{\SK}(\us) \Big\rangle_t \,,
$$
resulting with 
\begin{align*}
\left(\frac{\partial \phi_n}{\partial t} \right)_{\SK}
= -\frac{1}{n} \frac{\beta}{2\sqrt{t}} \E_{\mathbf{z}} \Big(
 \E_{\mathbf{J}} [ \langle H^{\SK}(\sigma) \rangle_t ] \Big) \,.
\end{align*}
The expression on the \abbr{rhs} of \eqref{eq:der-SK} then follows 
by an application of Gaussian integration by parts 
to $\E_{\mathbf{J}} [\langle H^{\SK}(\us) \rangle_t]$,
as illustrated for example in \cite[Lemma 1.1]{Pan}. 

Next, to establish \eqref{eq:der-D} let 
$h_{ij}(z_{ij}) := \E [\psi_n(t,\mathbf{z},\mathbf{J})|z_{ij}]$, 
and note that the product form of $\mathbf{p}(t,\mathbf{z})$ and 
$\mu(\mathbf{z},\mathbf{J})$, results with
\begin{align}\label{eq:der-D-temp}
\left(\frac{\partial \phi_n}{\partial t} \right)_{D} = 
\sum_{i=1}^{n} \sum_{j=1}^{n} \int h_{ij}(z) 
 \frac{\partial p}{\partial t}(t,z) \de\nu_{\N} (z) \,.
\end{align}
The $ij$-th integral on the \abbr{rhs} of \eqref{eq:der-D-temp}
is merely the value of $(-\gamma/n) g'(\lambda)$,
where $g(\lambda)=\E[f(z)]$ for $f=h_{ij}$ and $z \sim \dPois(\lambda)$
at $\lambda=\gamma(1-t)/n$. Differentiating the $\dPois(\lambda)$ 
\abbr{pmf} one has the identity 
$g'(\lambda) = \E[f(z+1)- f(z)]$ 
(under mild regularity conditions on $f$). This crucial observation
transforms \eqref{eq:der-D-temp} into
\begin{align}\label{eq:der-D-temp2}
\left(\frac{\partial \phi_n}{\partial t} \right)_{D} = -\frac{\gamma}{n} 
\sum_{i=1}^{n} \sum_{j=1}^{n} \E [h_{ij}(z_{ij}+1)- h_{ij}(z_{ij})] \,.
\end{align}
Here $\psi_n(t,\cdot,\cdot)=n^{-1} \log Z_n(\beta,\gamma,t)$ 
and adding one to $z_{ij}$ corresponds to an extra copy of the 
edge $(i,j)$ in the dilute Ising model of Hamiltonian 
$\frac{1}{\sqrt{2\gamma}} H_{\gamma(1-t)}^{\D}(\us)$. 
Consequently, setting $b:=\frac{\beta}{\sqrt{2\gamma}}$,
\begin{equation}\label{eq:h-diff}
h_{ij}(z_{ij}+1) - h_{ij}(z_{ij}) = \frac{1}{n} 
\log \Big\langle e^{b \sigma_i \sigma_j} \Big\rangle_t 
= \frac{1}{n} \log \Big\{ \cosh(b) 
\big[1 + \tanh(b) \langle \sigma_i \sigma_j \rangle_t \big] \Big\} \,, 
\end{equation}
since $e^{b y} = \cosh(b) [ 1+\tanh (b) y ]$
for the $\{-1,+1\}$-valued $y= \sigma_i\sigma_j$.
Combining \eqref{eq:der-D-temp2} and \eqref{eq:h-diff}, we
obtain by the Taylor series for $-\log(1+x)$ (when $-1<x<1$), that
\begin{align*}
\left(\frac{\partial \phi_n}{\partial t} \right)_{D} &=
-\frac{\gamma}{n^2} \sum_{i=1}^{n} \sum_{j=1}^{n}  \E \left[ 
\log \left\{ \cosh(b) \left[ 1+ \tanh(b)\langle \sigma_i \sigma_j \rangle_t \right] \right\}\right]\nonumber \\
&= -\gamma \log \cosh (b) + \gamma \sum_{\ell=1}^{\infty} \frac{(-1)^\ell}{\ell} \Big( \tanh (b) \Big)^\ell \E\Big[ \frac{1}{n^2} \sum_{i,j=1}^n ( \langle \sigma_i \sigma_j \rangle_t)^\ell \Big] \nonumber \\
&= -\gamma \log\cosh (b)+ \gamma \sum_{\ell=1}^{\infty} \frac{(-1)^\ell}{\ell} \Big(  \tanh (b) \Big)^\ell \E[\langle Q_{\ell}^2 \rangle_t ] \,,
\end{align*}
as stated in \eqref{eq:der-D}.

\section{Graph Comparison: Proof of Theorem \ref{d-reg}}
\label{comparison}

The notion of uniform random $\gamma$-regular graph refers to 
drawing such graph uniformly from among all $\gamma$-regular 
simple graphs on $n$-vertices, provided, as we assume 
throughout, that $n\gamma$ is even. 
We instead denote by ${\Rg \gamma}$ the more tractable 
configuration model, where each vertex is equipped with 
$\gamma$ half-edges and a multigraph (of possible 
self-loops and multiple edges) is formed by a 
uniform random matching of the collection of all 
$\gamma n$ half-edges. Indeed, as mentioned in the 
context of \ER graphs (see start of Section \ref{interpolation}),
for $\gamma$ bounded the matching in ${\Rg \gamma}$ produces 
a simple graph  with probability bounded away from zero, and
conditional on being simple this graph is uniformly random.
Consequently, any property that holds w.h.p. for the
configuration model multigraph ${\Rg \gamma}$ must 
also hold w.h.p. for the simple uniform random 
$\gamma$-regular graph.

Our strategy for proving Theorem \ref{d-reg} is to start from
the random regular multigraph $G_1 \sim {\Rg \gamma}$, deleting 
some edges and ``rewiring" some of the existing ones to obtain 
a new graph $G_2$ which is approximately an \ER random graph 
of $n \gamma_-/2$ edges, where 
$\gamma_-:= \gamma- \sqrt{\gamma}\log \gamma$.
Then, with Theorem \ref{ER} providing us with the typical 
behavior of extreme bisections of $G_2$, the main 
challenge is to control the effect of our edge transformations 
well enough to handle the minimum and maximum bisections 
of $G_1$.

Specifically, drawing i.i.d. $X_i \sim \dPois(\gamma_-)$, we 
let $Z_i := (\gamma-X_i)_+$ and color $Z_i$ of 
the $\gamma$ half-edges of each vertex $i \in [n]$ 
by {\small\sf BLUE} ({\small\sf B}). All other half-edges 
are colored {\small\sf RED} ({\sf R}). Matching the half-edges 
uniformly, without regard to their colors, we obtain a graph 
$G_1 \sim {\Rg \gamma}$. Our coloring decomposes $G_1$ to 
the sub-graph $\Grr$ consisting of all the {\small\sf RR} 
edges and $\Grb \cup \Gbb$ having all other edges, which we 
in turn decompose to the sub-graph $\Gbb$ consisting of the 
{\small\sf BB} edges and $\Grb$ having all the 
multi-color edges (i.e. {\small\sf RB} and {\small BR}).
To transform $G_1$ to $G_2$, we first delete all edges 
of $\Gbb$, disconnect all the multi-colored {\sf RB} edges 
and delete all the {\sf B} half-edges that as a result 
became unmatched. We then form a new sub-graph $\tGrr$ 
by uniformly re-matching all the free {\sf R} half-edges
(in case there is an odd number of such half-edges 
we leave one of them free as a self-loop). The
graph $G_2$ has the vertex set $[n]$ and 
$E(G_2) = E(\Grr) \cup E(\tGrr)$. 

\smallskip
We represent by $\Omega_n$ the collection of all 
bisections for a graph $G$ having $n$ vertices, denoting 
by $\cut_G(\us)$ the cut size for the partition between 
$\{i \in [n] : \sigma_i=-1\}$ and its complement. 
Then, for any $\us \in \Omega_n$ we have
\begin{align}
\cut_{G_1}(\us) &=\cut_{G_2}(\us)- \cut_{\tGrr}(\us) + \cut_{\Grb \cup \Gbb}(\us) 
\,. 
\label{cut_equation}
\end{align}
We control the \abbr{lhs} of \eqref{cut_equation} by three key lemmas,
starting with the following consequence of Theorem \ref{ER}, proved
in Subsection \ref{proof_grr} that gives sharp estimates
on the dominant part, namely $\cut_{G_2}(\us)$. 
\begin{lemma}
\label{Grr}
We have, w.h.p. as $n \to \infty$,
\begin{align}\label{eq:mcut-G2}
\frac{\mcut(G_2)}{n} &= \frac{\gamma_-}{4}- P_* \sqrt{\frac{\gamma}{4}} + o_{\gamma}(\sqrt{\gamma}) \, , \\
 \frac{\Mcut(G_2)}{n} &= \frac{\gamma_-}{4}+ P_* \sqrt{\frac{\gamma}{4}} + o_{\gamma}(\sqrt{\gamma}).\label{eq:Mcut-G2}
\end{align}
\end{lemma}

Our next lemma, proved in Subsection \ref{proof_expectedcut}, 
shows that while both the {\small\sf B} half-edge deletions and 
the {\small\sf R} half-edge re-matching that follows, may affect
the cut size, on the average (with respect to our random matching),
at the scale of interest to us they cancel out each other. 
\begin{lemma}
\label{expectedcut}
Uniformly over all $\us \in \Omega_n$, 
\begin{align}
\E[\cut_{\Grb}(\us)]&= n\left(\frac{\sqrt{\gamma}\log \gamma}{2} + O_{\gamma}(1)  \right) + o(n) \, ,\label{e1}\\
\E[\cut_{\tGrr}(\us)] &= n\left(\frac{\sqrt{\gamma}\log \gamma }{4} + O_{\gamma}(1) \right) + o(n)\, , \label{e2}\\
\E[\cut_{\Gbb}(\us)] &= n \left(\frac{(\log \gamma)^2}{4} + o_{\gamma}(1) \right) 
+ o(n).\label{e3}
\end{align}
\end{lemma}

The last result we need, is the following uniform bound on the
fluctuations, proved in Subsection \ref{proof_uniformbound}, that
allows us to control the effect of the edge rewiring
on the extremal bisections. 
\begin{lemma}
\label{uniform_bound}
There exists $C$ sufficiently large, independent of $n$ and $\gamma$, such that 
\begin{align}
\P\left[\sup_{\us \in \Omega_n} |\cut_{\mathcal{A}}(\us) - \E[\cut_{\mathcal{A}}(\us)]| > C n \gamma^{1/4} \sqrt{\log \gamma}  \right] &= o(1) \label{estimate_equation1}
\end{align}
where $\mathcal{A}$ may be distributed as $\Grb \cup \Gbb$ or  $\tGrr$. 
\end{lemma}
Turning to prove Theorem \ref{d-reg}, we have 
from \eqref{cut_equation} and Lemma \ref{uniform_bound} 
that w.h.p. as $n \to \infty$, 
%
\begin{align}\label{eq:bisec-reg}
\sup_{\us \in \Omega_{n}} \Big| \cut_{G_1}(\us) -\cut_{G_2}(\us) 
+ \E[\cut_{\tGrr}(\us)] - \E[\cut_{\Grb \cup \Gbb}(\us) ] \Big|
= n o_{\gamma}(\sqrt{\gamma}) \,.
\end{align}
In view of Lemma \ref{expectedcut}, we deduce from \eqref{eq:bisec-reg} that      
w.h.p. as $n \to \infty$,
$$
\sup_{\us \in \Omega_n} \Big| \cut_{G_1}(\us) -\cut_{G_2}(\us) 
- n \frac{\sqrt{\gamma} \log \gamma}{4} \Big|
= n o_{\gamma}(\sqrt{\gamma}) + o(n) \,.
$$
This in turn implies that w.h.p.
\begin{align*}
\mcut(G_1) &= \mcut(G_2) + n \frac{\sqrt{\gamma} \log \gamma}{4} + 
n o_{\gamma}(\sqrt{\gamma}) + o(n) \,, \\
\Mcut(G_1) &= \Mcut(G_2) + n \frac{\sqrt{\gamma} \log \gamma}{4} + 
n o_{\gamma}(\sqrt{\gamma}) + o(n) \,,
\end{align*} 
and Theorem \ref{d-reg} thus follows from Lemma \ref{Grr}
(recall that $\gamma=\gamma_- + \sqrt{\gamma} \log \gamma$).

\subsection{Proof of Lemma \ref{Grr}}
\label{proof_grr}
Let $G_n^{\rm{int}}$ be the random graph generated from the 
configuration model with i.i.d. $X_i \sim \dPois(\gamma_-)$ 
degrees. We denote by $\Pc {\gamma_-}$ the sub-graph 
obtained by independently deleting each half-edge 
of $G_n^{\rm{int}}$ with probability $1/n$, before 
matching them. By the thinning property of the $\dPois$ law, 
$\Pc {\gamma_-}$ has the law of the Poisson-Cloning model, where 
one first generates i.i.d. $\zeta_i \sim \dPois(\frac{n-1}{n}\gamma_-)$, 
then draws a random graph from the configuration model with $\zeta_i$ 
half-edges at vertex $i$. Recall \cite{Kim} that  
the $G_I(n, \frac{\gamma_-}{n})$ and $\Pc {\gamma_-}$ models are mutually contiguous.
Further, $\gamma_-/\gamma \to 1$, and so by Theorem \ref{ER}, w.h.p. 
\begin{align}\label{eq:mcut-Pc}
\frac{\mcut(\Pc{\gamma_-})}{n} &= \frac{\gamma_-}{4}- P_* \sqrt{\frac{\gamma}{4}} + o_{\gamma}(\sqrt{\gamma})\, , \\
 \frac{\Mcut(\Pc{\gamma_-})}{n} &= \frac{\gamma_-}{4}+ P_* \sqrt{\frac{\gamma}{4}} + o_{\gamma}(\sqrt{\gamma})\,.\label{eq:Mcut-Pc}
\end{align}
Next note that for any two graphs $\mathcal{G}_1, \mathcal{G}_2$ on $n$ vertices, 
$|\Mcut(\mathcal{G}_1) - \Mcut(\mathcal{G}_2| \leq |E(\mathcal{G}_1) \Delta E(\mathcal{G}_2)|$ and $|\mcut(\mathcal{G}_1) - \mcut(\mathcal{G}_2| \leq |E(\mathcal{G}_1) \Delta E(\mathcal{G}_2)|$. W.h.p. our coupling has 
$\sum_i (X_i-\zeta_i) = O(1)$ half-edges from $G_n^{\rm{int}}$ 
not also in $\Pc {\gamma_-}$. 
Hence $|E(G_n^{\rm{int}}) \Delta E(\Pc {\gamma_-})| = O(1)$ 
and \eqref{eq:mcut-Pc}-\eqref{eq:Mcut-Pc} extend to 
$\mcut(G_n^{\rm{int}})$ and $\Mcut(G_n^{\rm{int}})$, respectively.

We proceed to couple $G_n^{\rm{int}}$ and $G_2$ such that  
$|E(G_n^{\rm{int}}) \Delta E(G_2)| \leq n o_{\gamma}(\sqrt{\gamma})$ w.h.p.
thereby yielding the desired conclusion.  
To this end, $G_2$ could have alternatively been generated by \emph{one} 
uniform random matching of only the $X_i' := \min\{X_i,\gamma\}$ 
{\small\sf RED} half-edges that each vertex $i$ has in $G_1$ 
(for completeness, we prove this statement in Lemma \ref{lem-pair}). 
We can thus couple $G_2$ and $G_n^{\rm{int}}$ by first forming 
$G_n^{\rm{int}}$, then independently for $i=1,\ldots,n$ 
color in {\small\sf RED} uniformly at random $X_i'$ of the 
$X_i$ half-edges of vertex $i$, with all remaining half-edges 
colored {\small\sf BROWN}. Now, to get $G_2$ we delete all
{\small\sf BB} edges, disconnect all {\small\sf RB} edges
and delete the resulting {\small\sf B} half-edges, then 
uniformly re-match all the free {\small\sf R} half edges
(for Lemma \ref{lem-pair} applies again in this setting).
The claimed bound on $|E(G_n^{\rm{int}}) \Delta E(G_2)|$ 
follows since the total number of {\small\sf B} half-edges 
in $G_n^{\rm{int}}$ is w.h.p. 
at most
\begin{equation}\label{eq:norm-aprox}
2 n \E (X_1- X_1') = 2 n \E[(X_1-\gamma)_+] = n O_{\gamma}(1) \,,
\end{equation}
where the \abbr{rhs} follows by Normal approximation to 
$\dPois(\gamma_-)$ and our choice of 
$\gamma_- = \gamma - \sqrt{\gamma} \log \gamma$.
\subsection{Proof of Lemma \ref{expectedcut}}
\label{proof_expectedcut}
We first prove \eqref{e1}, utilizing the fact that the distribution of $\cut_{\Grb}(\us)$ is the same for all $\us \in \Omega_n$. Hence,
\begin{align}
\E[\cut_{\Grb}(\us)] = \E\big[\,\E_{\us^\star}[\cut_{\Grb}(\us^\star)]\,\big] \label{symmetry}
\end{align}
for $\us^\star$ chosen uniformly from $\Omega_n$. Given the graph $G_1$, we have
 \begin{align}\label{e-cond}
\E_{\us^\star}[\cut_{\Grb}(\us^\star)] = \frac{|\Erb|}{2(1-1/n)} \,,
\end{align}
where $\Erb$ denotes the set of {\small\sf RB} edges in $G_1$ excluding self-loops. 
Next, noting that the expected number of edges in $G_1$ excluding self-loops is 
$\frac{n(n-1)\gamma^2}{2(n\gamma-1)}$ and the probability that an edge connecting
two distinct vertices is coloured {\small\sf RB} is $2\:\frac{\E[Z_1]}{\gamma}(1- \frac{\E[Z_1]}{\gamma})$, we have, 
\begin{align}
\E\big[\,|\Erb|\,\big] = \frac{n(n-1)\gamma^2}{n\gamma-1} \frac{\E[Z_1]}{\gamma} \Big( 1- \frac{\E[Z_1]}{\gamma} \Big) \,, \label{e4}
\end{align}
where $Z_1 \sim (\gamma-X_1)_+$ and $X_1 \sim \dPois(\gamma_-)$. 
We get \eqref{e1} out of \eqref{symmetry} and \eqref{e4} upon observing that 
\begin{align}\label{eq:exp-z1}
\E[Z_1] &= \gamma-\E[X_1] + \E[(\gamma-X_1)_-] 
         = \gamma-\gamma_- + \E[(X_1-\gamma)_+]
         = \sqrt{\gamma}\log \gamma + O_{\gamma}(1) \,
\end{align}
(see \eqref{eq:norm-aprox} for the right-most identity). 
By an analogous calculation, we find that for all $\us \in \Omega_n$, 
\begin{align*}
\E\big[\cut_{\Gbb}(\us)\big] = \frac{n(n-1)\gamma^2}{2(n\gamma-1)} \Big(
\frac{\E[Z_1]}{\gamma} \Big)^2 \frac{1}{2(1-1/n)}  = n \big[
\frac{1}{4} (\log \gamma)^2 + o_{\gamma}(1)\, \big] + o(n) .
\end{align*}
Turning to \eqref{e2}, the same argument as in \eqref{symmetry} implies that
\begin{align*} 
\E\big[
\cut_{\tGrr}(\us)\big] = \E\big[\,\E_{\us^\star}[\cut_{\tGrr}(\us^\star)]\,\big]\,,
\end{align*}
for $\us^\star$ chosen unifomly from $\Omega_n$. Further, similarly to 
\eqref{e-cond} we find that given the graph $\tGrr$, 
\begin{equation}\label{e-cond2}
\E_{\us^\star}\big[\,\cut_{\tGrr}(\us^\star)\,\big] = \frac{|E_2|}{2(1-1/n)} \,, 
\end{equation}
where $E_2$ denotes the set of edges in $\tGrr$ excluding self-loops. Recall 
that $|E(\Grb)| - |\Erb|$ and $\frac{1}{2} |E(\Grb)|-|E_2|$ count the 
number of self-loops in $\Grb$ and $\tGrr$, respectively. The expected
number of such self-loops is $O(1)$ as $n \to \infty$, hence
$\E[\,|E_2|\,] = \frac{1}{2} \E[\, |\Erb| \, ] + O(1)$, which upon 
comparing \eqref{e-cond} to \eqref{e-cond2} yields the required expression 
of \eqref{e2}.  
\subsection{Proof of Lemma \ref{uniform_bound}}
\label{proof_uniformbound}
Starting with $\mathcal{A} = \Grb \cup \Gbb$ clearly, for any $x_n > 0$, 
\begin{align}\label{eq:basic-decom}
\P\Big[ \sup_{\us \in \Omega_n} \big|\cut_{\mathcal{A}}(\us) - \E[\cut_{\mathcal{A}}(\us)]\big| \geq 2 x_n\Big]  \leq p_1(n) + p_2(n) 
\end{align}
where $\Deg = (Z_1, \cdots, Z_n)$ count the 
number of {\small\sf BLUE} half-edges at each vertex of $G_1$ and 
\begin{align}
p_1(n) &= \P\Big[ \sup_{\us \in \Omega_n} \big|\cut_{\mathcal{A}}(\us) - 
c(\us,\Deg) \big| \geq x_n \Big] \label{concentration_terms1} \,,\\ 
p_2(n) &= \P\Big[ \sup_{\us \in \Omega_n} \big| c(\us,\Deg) - \E[\cut_{\mathcal{A}}(\us)] \big| \geq x_n \Big]\,, \label{concentration_terms2}
\end{align}
for $c(\us,\Deg):=\E[\cut_{\mathcal{A}}(\us) | \Deg]$.
Letting $S_n(\Deg)=\sum_{i=1}^n Z_i$, note that w.h.p. $\Deg \in \mathcal{E}_n$ for  
$\mathcal{E}_n = \{ \mathbf{z} : |S_n(\mathbf{z}) -n  \E [Z_1] \,| \le b_n \}$ 
and $b_n = \sqrt{n\log n}$.
Hence, by a union bound over $\us \in \Omega_n$ we get that
\begin{align}
p_1(n) &\leq 2^n \max_{\mathbf{z} \in \mathcal{E}_n}
 \max_{\us \in \Omega_n} \P\left[
         \left|\cut_{\mathcal{A}}(\us) -
         c(\us,\Deg)\right| \geq
         x_n \Big| \Deg = \mathbf{z} \right]   + o(1)\, .
         \label{concentration_first} 
\end{align}
We next apply Azuma-Hoeffding inequality to 
control the \abbr{rhs} of \eqref{concentration_first}.
To this end, fixing $\mathbf{z} \in \mathcal{E}_n$ and 
half-edge colors such that $\{\Deg = \mathbf{z}\}$, 
we form $G_1$ by sequentially pairing a candidate half-edge 
to uniformly chosen second half-edge, using first 
{\small\sf BLUE} half-edges as candidates for the 
pairing (till all of them are exhausted).
Then, fixing $\us \in \Omega_n$, we
consider Doob's martingale $M_k=\E[ \cut_{\mathcal{A}}(\us) | \mathscr{F}_k]$,
for the sigma-algebra $\mathscr{F}_k$ generated by all half-edge colors and 
the first $k \ge 0$ edges to have been paired. This martingale 
starts at $M_0=c(\us,\mathbf{Z})$, has differences $|M_k-M_{k-1}|$ 
uniformly bounded by some universal finite non-random 
constant $\kappa$ (independent on $n$, $\us$ and $\mathbf{z}$),
while $M_\ell=\cut_{\mathcal{A}}(\us)$ for all $\ell \ge S_n(\mathbf{z})$
(since the sub-graph $\mathcal{A} = \Grb \cup \Gbb$ is 
completely formed within our sequential matching 
first $S_n(\mathbf{z})$ steps). The bounded difference property of $M_k$ follows easily from the ``switching" argument in \cite[Theorem 2.19]{wormald}. Thus, from Azuma-Hoeffding 
inequality we get that for $\mathbf{z} \in \mathcal{E}_n$,
\begin{align}
\P\left[ \left|\cut_{\mathcal{A}}(\us) - c(\us,\Deg) \right| \geq x_n \Big| \Deg =\mathbf{z} \right]   
\leq 2\,\exp\left(-\frac{x_n^2}{8 \kappa^2  S_n(\mathbf{z})} \right)
\leq 2\, \exp\left(-\frac{x_n^2}{8 \kappa^2 ( n \E[Z_1]+ b_n)} \right) \label{azuma} 
\end{align}
Recall \eqref{eq:exp-z1} 
that $\E[Z_1] = \sqrt{\gamma} \log \gamma + O_\gamma(1)$, hence  
choosing $x_n = C n \gamma^{1/4}\sqrt{\log \gamma}$ for some
$C^2 > 8\kappa^2 \log 3$, we find that the \abbr{rhs} of  
\eqref{concentration_first} decays to zero as $n \to \infty$.  

Turning to control $p_2(n)$, for $i \in [n]$ and $1 \le j \le Z_i$, 
let $I_{ij}(\us)= 1$ if the $j^{\rm {th}}$ {\small\sf {B}} half-edge 
of vertex $i$ is matched to some half-edge from the opposite side of 
the partition induced by $\us$, and $I_{ij}(\us)=0$ otherwise. Then,
\begin{align}
\cut_{\mathcal{A}}(\us) = \sum_{i=1}^{n} \sum_{j=1}^{Z_i} I_{ij} (\us) - \cut_{\Gbb}(\us) .\label{indicators}
\end{align}
For $i$ such that $\sigma_i=1$ and 
$1\leq j \leq Z_i$ we similarly set $I'_{ij}(\us) = 1$ if the 
$j^{\rm{th}}$ {\small\sf{B}} half-edge of vertex $i$ is matched to a 
${\small\sf{B}}$ half-edge of a vertex from the opposite side,
and $I'_{ij}=0$ otherwise. Clearly then 
$$
\cut_{\Gbb} (\us) = \sum_{\{i: \sigma_i =1\}} \sum_{j=1}^{Z_i} I'_{ij}(\us)\,,
$$
so setting $S_n^+(\us,\Deg):=\sum_{\{i:\sigma_i=1\}} Z_i$, we have
from \eqref{indicators} that 
\begin{align}
c(\us,\Deg)  &= \sum_{i=1}^{n} \sum_{j=1}^{Z_i} \P[I_{ij}(\us)=1|\Deg] - 
\sum_{\{i:\sigma_i =1\}} \sum_{j=1}^{Z_i} \P[I'_{ij}(\us) = 1 | \Deg] \nonumber \\
&= S_n(\Deg)  \frac{(n\gamma)/2}{n\gamma -1} - 
S_n^+(\us,\Deg)  \, \frac{S_n(\Deg) - S_n^+(\us,\Deg)}{n\gamma -1} \,. \label{conditionalmean}
\end{align} 
Considering the extreme values of the \abbr{rhs} of \eqref{conditionalmean} 
yields that for all $\us \in \Omega_n$, 
\begin{align*}
\frac{1}{2} S_n(\Deg) \Big(1 -
\frac{S_n(\Deg)}{2 n \gamma} \Big) \leq c(\us,\Deg) \Big(1-\frac{1}{n \gamma}\Big)  
\leq
\frac{1}{2} S_n(\Deg) \,,
\end{align*}
from which we deduce that 
\begin{align}\label{sandwich}
n \frac{\E[Z_1]}{2} \Big(1 -
\frac{\E[Z_1]}{2 \gamma} \Big) + o(n) 
&\leq \inf_{\Deg \in \mathcal{E}_n} \inf_{\us \in \Omega_n} \{c(\us,\Deg)\}   
\nonumber \\
&\leq \sup_{\Deg \in \mathcal{E}_n} \sup_{\us \in \Omega_n} \{c(\us,\Deg)\} \leq
n \frac{\E[Z_1]}{2} + o(n) \,.
\end{align}
Further, while proving Lemma \ref{expectedcut} we have shown that 
$$
\E[\cut_{\mathcal A}(\us)] = n \frac{\E[Z_1]}{2} \frac{n\gamma}{(n\gamma-1)} 
 \Big(1-\frac{\E[Z_1]}{2\gamma}\Big) 
\,,
$$
hence from \eqref{sandwich} and \eqref{eq:exp-z1} it follows that 
$$
\sup_{\Deg \in \mathcal{E}_n} \sup_{\us \in \Omega_n} \big| \, c(\us,\Deg) 
- \E[\cut_{\mathcal A}(\us)]\, \big| \le n \frac{\E[Z_1]^2}{4 \gamma} + o(n) 
\le n (\log \gamma)^2 + o(n) \,,
$$
and since w.h.p. $\Deg \in \mathcal{E}_n$, we conclude that $p_2(n) = o(1)$. 

Next we consider the graph $\mathcal{A}=\tGrr$ and proceeding in a similar 
manner we have the decomposition \eqref{eq:basic-decom}, except for replacing 
in this case $\Deg$ in \eqref{concentration_terms1}-\eqref{concentration_terms2} by the count 
$\Degsec = (Y_1, Y_2, \cdots, Y_n)$ of number of {\small\sf R} 
half-edges at each vertex at the initiation of the second stage. 
The total number $S_n(\Degsec)$ of {\small\sf R} half-edges to be matched 
in the second step is less than the initial number $S_n(\Deg)$ of 
{\small \sf B} half-edges. Consequently, if $\Deg \in \mathcal{E}_n$ then
$\Degsec \in \mathcal{E}_n^+ := 
\{ \mathbf{y} : S_n(\mathbf{y}) \leq n \E[Z_1] + b_n \}$, and we
have again a bound of the type \eqref{concentration_first} on $p_1(n)$,
just taking here the maximum over $\mathbf{y} \in \mathcal{E}_n^+$ instead
of $\mathbf{z} \in \mathcal{E}_n$. Further, we repeat the martingale
construction that resulted with the \abbr{rhs} of \eqref{azuma}. Specifically, 
here $\mathcal{F}_0$ is the sigma-algebra of $\Degsec$, namely
knowing the degrees of vertices in $\tGrr$, and we 
expose in $\mathcal{F}_k$ the first $k$ edges to have been 
paired en-route to the uniform matching that forms $\tGrr$.
As before the Doob's martingale $M_k=\E[\cut_{\tGrr}(\us)|\mathcal{F}_k]$
has uniformly bounded differences, starting at $M_0=c(\us,\Degsec)$ and 
with the same choice of $x_n$ the desired bound on $p_1(n)$ 
follows upon observing that 
$M_\ell=\cut_{\tGrr}(\us)$ for all $\ell \ge S_n(\mathbf{y})$, hence
as soon as $\ell=n \E[Z_1]+b_n$. 
Turning to deal with $p_2(n)$ in this context, 
by the same reasoning that led to \eqref{conditionalmean} we
find that for $S=S_n(\mathbf{y}) \ge 2$ and $S^+(\us)=S_n^+(\us,\mathbf{y})$,
\begin{align}\label{eq:c-red-formula}
c(\us,\mathbf{y}) = \E[\cut_{\tGrr}(\us)|\Degsec=\mathbf{y}] = 
\frac{S^{+}(\us) (S- S^{+}(\us))}{S-1} 
 = \frac{S^2 - \big(2S^+(\us)-S\big)^2}{4(S-1)} \,. 
\end{align}
While proving Lemma \ref{expectedcut} we have shown that
$$
\overline{c}_n :=  4 \E[\cut_{\tGrr}(\us)] =
 n \big[ \E[Z_1] + O_\gamma(1) \big] + o(n)
$$
and that $\overline{c}_n$ is constant over $\us \in \Omega_n$. 
As $\overline{c}_n \ge 6 x_n$ for $n$ and $\gamma$ 
large, while $|S^2/(S-1) - S|$ is uniformly bounded,
we deduce from \eqref{eq:c-red-formula} that for any $y_n \le x_n$, 
$$
\{ |2S^+(\sigma)-S'|< x_n\} \bigcap \{ | S - S'| < y_n \} 
\bigcap \{ |S'- \overline{c}_n| < x_n \}
\;\; \Longrightarrow \;\; \{ |4 c(\us,\mathbf{y}) - \overline{c}_n| < 4 x_n \} \,.
$$ 
Now, w.h.p. $S'=S_n(\Deg)$ is in 
$\mathcal{I}_n := [n \E[Z_1] - b_n, n \E[Z_1] +  b_n]$
with $|S' - \overline{c}_n| < x_n$, and taking then 
the union over $\us \in \Omega_n$, we get similarly to the derivation 
of \eqref{concentration_first} that
\begin{align}
p_2(n) \leq &2^n \max_{s_n \in \mathcal{I}_n}  
\max_{\us \in \Omega_n} \P\Big[ \, |  2 S^+(\us) - s_n
| \ge  x_n \,, S > s_n - y_n \, \big| \, S' = s_n \Big] 
\nonumber \\
& + 
\max_{s_n \in \mathcal{I}_n} 
\P\big( S \le s_n - y_n\,\big|\, S'=s_n \big) + o(1) 
=: p_3(n) + p_4(n) + o(1) \, .
\label{tgrr:second}
\end{align}
Starting with $2N=n\gamma$ half-edges of $G_1$ of whom $S'=s_n$ 
colored {\small\sf B} (while all others are colored 
{\sf\small R}), the non-negative number $S'-S$ 
of half-edges in $\Gbb$ is stochastically 
dominated by a $\dBin (s_n,s_n/(2N-s_n))$ random variable.
For $s_n \in \mathcal{I}_n$ the latter Binomial has mean 
$n \E[Z_1]^2/(\gamma-\E[Z_1]) + o(n)$, hence 
in view of \eqref{eq:exp-z1},
$p_4(n) = o(1)$ provided $y_n \ge 3 n (\log \gamma)^2$.
For bounding $p_3(n)$ we assume w.l.o.g. that $\sigma_i =1$ iff $i \le n/2$,
with $S^{+}(\us)$ the total number of {\small\sf R} half-edges 
for vertices $i \le n/2$, which are matched to {\small\sf B} 
half-edges by the uniform matching in the first step. Fixing 
the total number $s_n$ of {\small\sf B} half-edges in $G_1$, 
clearly $S^+(\us)$ is stochastically decreasing in the number 
$S'_+ = \sum_{i=1}^{n/2} Z_i$ of
{\small\sf B} half-edges among vertices $i \leq n/2$. 
Thus, it suffices to bound $p_3(n)$ in the extreme 
cases, of $S'_+=0$, and of $S'_+ = s_n$. The uniform matching 
of the first step induces a sampling without replacement 
with $S^{+}(\us)$ denoting the number of marked balls when drawing a 
sample of (random) size $S \in (s_n-y_n,s_n]$, uniformly without 
replacement from an urn containing $2N-s_n$ balls, 
of which either $N$ or $N-s_n$ balls are marked. By 
stochastic monotonicity, it suffices to consider 
the relevant tails of $S^+(\us)-s_n/2$ only in the 
extreme cases of $S=s_n-y_n$ and $S=s_n$.
As $2N/n = \gamma$, $s_n/n \ll \gamma$ and $y_n \ll x_n$, 
standard tail bounds for 
the hyper-geometric distribution \cite{hypergeometric}  
imply that $p_3(n) =o(1)$ for $\gamma$ sufficiently large, 
thereby completing the proof.

\subsection{A pairing lemma} 
\label{pairing}
We include here, for completeness, the formal proof 
of the fidelity of the two stage pairing procedure
(which was used in our preceding arguments).
\begin{lemma}\label{lem-pair}
Given $2\ell$ labeled balls of color {\small\sf R} and $2m$ 
labeled balls of color {\small \sf B} for some $m \le \ell$,
we get a uniform random pairing of the {\small\sf R} balls 
by the following two step procedure:
\begin{enumerate}
\item First match the $2(m+\ell)$ balls uniformly at random to obtain 
some {\small\sf RR}, {\small\sf RB} and {\small\sf BB} pairs.
\item Remove all {\small\sf B} balls and uniformly re-match the {\small\sf R}
balls which were left unmatched due to the removal of the {\small\sf B} balls. 
\end{enumerate} 
\end{lemma}
\begin{proof} 
We use the notation 
$(2k-1)!! = (2k-1) (2k-3) \cdots 1$ and $[m]_k= m(m-1) \cdots (m-k+1)$ and let
$\mathscr{P}$ denote the random pairing of the $2\ell$ {\small\sf R} balls 
by our two-stage procedure (which first generated $2s$ pairs of type 
{\small\sf RB}, 
$(m-s)$ of type {\small\sf BB} and $(\ell-s)$ of type {\small\sf RR}).
We then have that for any fixed final pairing $P$ of the {\small \sf R} balls, 
\begin{align}
\P[ \mathscr{P}= P] &= \sum_{s=0}^{m} {\ell \choose s} \frac{[2m]_{2s} (2(m-s)-1)!!}{(2m+2\ell-1)!! (2s-1)!!} \nonumber\\
&= \frac{\ell! (2m)! 2^\ell}{ (2m+2\ell)!} 
\sum_{s=0}^{m} 2^{2s} { m+\ell \choose {m-s, \ell-s, 2s}} 
= \frac{1}{(2\ell-1)!!} \,, \nonumber
\end{align}
where the last identity follows upon observing that 
$\sum_{s=0}^{m} 2^{2s} {m+\ell \choose {m-s, \ell-s, 2s}} = {2(m+\ell) \choose 2\ell}$.
\end{proof}
\section{From Bisection to Cut: Proof of Theorem \ref{maxcut_thm}}
\label{maxcut} 
Let $\mathcal{I}^\pm (\us) := \{ i : \sigma_i = \pm 1 \}$ be the 
partition of $[n]$ induced by $\us$ and 
$m(\us) := \frac{1}{2} \sum_{i=1}^n \sigma_i$
the difference in size of its two sides. Note that 
by the invariance of $\cut_G(\us)$ under the symmetry $\us \to -\us$, 
it suffices to compare the cuts in  $\mathcal{S}_n^+ = \{ \us \in \{-1,+1\}^n : 
m(\us)\geq 0 \}$ to those in $\Omega_n$. To this end, define the map 
$T: \mathcal{S}_n^+ \to \Omega_n$ where 
we flip the spins at the subset $V(\us)$ of smallest $m(\us)$ indices within 
$\mathcal{I}^+(\us)$, thereby moving all those indices to 
$\mathcal{I}^-(T(\us))$. Let 
$X(\us)$, $Y(\us)$ and $Z(\us)$ count the number of edges from $V(\us)$ 
to $\mathcal{I}^-(\us)$, $\mathcal{I}^-(T(\us))$
and $\mathcal{I}^+(T(\us))=\mathcal{I}^+(\us) \backslash V(\us)$,
respectively. Fixing $0<\delta< 1/4$ let 
\begin{align}\label{dfn:S-star}
\mathcal{S}^\star = \left\{ \us \in \mathcal{S}_n^+ : m(\us) \leq
                  \gamma^{-\delta} n \right\} \,.
\end{align}
Then, for $\us^\star \in \mathcal{S}_n^+$ such that 
$\MaxCut(G_n) = \cut_{G_n}(\us^\star)$ we have 
$$
\MaxCut(G_n) = \cut_{G_n} (T(\us^\star)) + X(\us^\star) - Z(\us^\star)
\le \Mcut(G_n) + Y(\us^\star) - Z(\us^\star) \,.
$$
Considering the union over $\us \in \mathcal{S}^\star$, we get that 
\begin{align}
 \P[ \MaxCut(G_n) > \Mcut(G_n) + \Delta_n]
 &\leq 
 2^n \max_{\us \in \mathcal{S}^\star}  
\,\P\Big[ Y(\us) - Z(\us) > \Delta_n \Big] 
+ \P\left[\us^\star \notin \mathcal{S}^\star\right]  
\nonumber \\
&=: q_1(n) + q_2(n) \, .  \label{terms}
\end{align} 
In proving part (a) of Theorem \ref{maxcut_thm}, we consider w.l.o.g. 
the \ER random graphs $G_n \sim G_I(n,\frac{\gamma}{n})$ as in 
Remark \ref{er_equivalence}. For fixed $\us \in \mathcal{S}^+_n$ each of
the independent variables $Y(\us)$ and $Z(\us)$ is $\dBin(N,\gamma/n)$
for $N=m(\us) (n/2)$. Upon computing the m.g.f. of 
$Y(\us)-Z(\us)$ we get by Markov's inequality that for any $\theta>0$,
$$
\P\Big[ Y(\us) - Z(\us) > \Delta_n \Big] \le e^{- 2 \theta \Delta_n} 
[1 + \frac{4 \gamma}{n} \sinh^2(\theta)]^N \,.
$$
Setting $\Delta_n =  n \gamma^{\psi/2}$ for some $\psi \in (1-\delta,1)$ fixed
and the maximal $N=\frac{1}{2} n^2 \gamma^{-\delta}$ for 
$\us \in \mathcal{S}^\star$, we deduce that
\begin{equation}\label{eq:ld-ge-thm}
\limsup_{n \to \infty}
n^{-1} \log \P\Big[ Y(\us) - Z(\us) > \Delta_n \Big] 
\le - 2 [\theta \gamma^{\psi/2} - \gamma^{1-\delta} \sinh^2(\theta)] =: - J \,.
\end{equation}
Since $\psi>1-\delta$ we have that  
$\gamma^{1-\delta} \sinh^2(\gamma^{-\psi/2}) \to 0$, 
so taking $\theta= \gamma^{-\psi/2}$ results with $J>1$ for all 
$\gamma$ large enough, in which case $q_1(n)=o(1)$ (see \eqref{terms}). 
As for controlling $q_2(n)$, recall Theorem \ref{ER} that w.h.p. 
$\MaxCut(G_n) \ge \Mcut(G_n) \ge n\gamma/4$. Hence, considering the 
union over $\us \notin \mathcal{S}^\star$ we have that 
$$
q_2(n) \le 2^n \max_{\us \notin \mathcal{S}^\star} 
\P\big( \, \cut_{G_n} (\us) \ge \frac{n \gamma}{4} \, \big) \,.
$$
For our \ER graphs $\cut_{G_n}(\us) \sim R_{k} :=
\dBin(k(n-k),\frac{\gamma}{n})$ with $k = \frac{n}{2} - m(\us)$. Taking the maximal 
$k^\star := \frac{n}{2} - n \gamma^{-\delta}$ for $\us \notin \mathcal{S}^\star$
and computing the relevant m.g.f. yields, similarly to \eqref{eq:ld-ge-thm},
that for $f_1(\theta)=e^\theta-1$, $f_2(\theta)=e^{\theta}-\theta-1$
and any $\theta>0$, 
\begin{equation}\label{eq:second-ge} 
\limsup_{n \to \infty}
n^{-1} \log \P
\big( \, R_{k^\star} \ge \frac{n \gamma}{4} \, \big)  
\le \frac{\gamma}{4} f_2(\theta) - \gamma^{1-2\delta} f_1(\theta) := -J' \,. 
\end{equation}
Since $\gamma f_2(\gamma^{-1/2})$ is uniformly bounded while 
$\gamma^{1-2\delta} f_1(\gamma^{-1/2}) = O_\gamma(\gamma^{1/2-2\delta})$
diverges (due to our choice of $\delta<1/4$), it follows that 
for $\theta=\gamma^{-1/2}$ and $\gamma$ large enough,
$J' \ge 1$ hence $q_2(n)=o(1)$, thereby completing the proof. 

The \ER nature of the graph $G_n$ is only used for deriving 
the large deviation bounds \eqref{eq:ld-ge-thm} and \eqref{eq:second-ge}. 
While slightly more complicated, similar computations apply also 
for $\Rg \gamma$. Indeed, in this case $Y(\us)-Z(\us)$ corresponds
to the sum of spins in a random sample of size $\gamma m(\us)$ taken   
without replacement from a balanced population of $\gamma n$ spins
(so by standard tail estimates for the hyper-geometric law, here too
the \abbr{lhs} of \eqref{eq:ld-ge-thm} is at most $-1$ for any 
$\gamma$ large enough). Similarly, now $R_{k^\star}$ counts the  
pairs formed by uniform matching of $\gamma n$ items, between 
a fixed set of $\gamma k^{\star}$ items 
and its complement (so by arguments similar to those we used
when proving  Lemma \ref{uniform_bound}, the \abbr{lhs} of \eqref{eq:second-ge} 
is again at most $-1$ for large $\gamma$). With the rest of the 
proof unchanged, we omit its details.
%
\bibliographystyle{amsalpha-nodash}
\bibliography{graphref}

\end{document}